\documentclass[8pt,a4paper]{amsart}
\usepackage[usenames,dvipsnames]{color}
\usepackage{amssymb} 
\usepackage{amsmath}
\usepackage{amsthm}
\usepackage[english]{babel}
\usepackage{verbatim}

\allowdisplaybreaks
\newtheorem{theorem}{Theorem}[section]
\newtheorem{lem}[theorem]{Lemma}

\newtheorem{defi}[theorem]{Definition}

\newtheorem*{thm}{Theorem}

\usepackage{enumerate}   
\usepackage{verbatim}

\usepackage{hyperref}
\hypersetup{
    colorlinks=true,
    citecolor=blue,
    filecolor=black,
    linkcolor=red,
    urlcolor=ForestGreen,
		linktoc=page
}

\numberwithin{equation}{section}

\begin{document}
\title
{Stability analysis of isometric embeddings}
\author{Norman Zergaenge}

\date{\today}

\begin{abstract}
In this work we prove the fact that, for a short time, it is possible to construct a smooth parametrized family of isometric embeddings of 
an arbitrary smooth parametrized family of Riemannian metrics on a smooth closed manifold into an Euclidean space. 
In order to prove this statement we work out stability estimates within the local perturbation method in \cite[Section 5]{gunther1989einbettungssatz} to derive a time-dependent local perturbation method around a free isometric embedding.
Iteratively, we use this time-dependent local perturbation method to construct the desired family of isometric embeddings. 

\end{abstract}

\maketitle 
\tableofcontents

\section{Introduction and statement of results}

A very fundamental problem in Riemannian geometry is whether it is possible to regard a given closed Riemannian manifold $(M^n,g)$ as a submanifold
of an Euclidean space. That is to say: exists a sufficiently smooth embedding $F: M\longrightarrow \mathbb{R}^N$ so that $F^{\ast}(\delta)=g$ 
where $F^{\ast}(\delta)$ denotes the pullback of the standard metric?
In case of existence one is naturally interested in a codimension $N-n$ which is as small as possible and one wishes to have a mapping $F_0$ which
is as regular as possible. It is worth noting that the \textit{isometry condition} $F^{\ast}(\delta)=g$ is equivalent to
\begin{equation*}
g_{ij}=\partial_i F\cdot \partial_j F=\sum_{l=1}^n{\partial_i F^l\cdot \partial_j F^l}\hspace{0.5cm}\text{for all }1\leq i\leq j\leq n
\end{equation*}
in local coordinates, which is a system of $n(n+1)/2$ partial differential equations of first order.

Without any doubt \textit{the breakthrough} on this field, apart from local precursors (cf. \cite{schlaefli1871nota}, \cite{janet1926}, \cite{cartan1927}), are the \textit{isometric embedding theorems} by John Nash in \cite{nash1954c}, \cite{nash1956imbedding} and \cite{nash1966analyticity}. The latter two works contain 
the statement that each $n$-dimensional
closed Riemannian manifold (of class $C^k$ where $3\leq k\leq \infty$ or analytic) has an isometric embedding (also of class $C^k$ where $3\leq k\leq \infty$ or analytic) into the space $\mathbb{R}^{N(n)}$ where $N(n)=n(3n+11)/2$. Over the years, there have been efforts to reduce the dimension $N(n)$, 
where we refer to \cite{gromov1986partial} for example.

At the end of the 1980s, Matthias G\"unther has released a method which allows to embed a smooth Riemannian manifold isometrically into the
Euclidean space of dimension $q(n):=\max\{n(n+5)/2,n(n+3)/2+5\}$ (cf. \cite[Satz 1.2]{gunther1989einbettungssatz}).
Similar to J. Nash in \cite{nash1956imbedding}, M. G\"unther has constructed the desired isometric embedding around a such called 
\textit{free embedding} (cf. Definition \ref{freemap}) which is related to the given metric $g$ in some sense. 
In this work we focus our attention on M. G\"unther's \textit{modification technique}, which yields the follows result:

\begin{thm}(cf. \cite[Satz 1.1]{gunther1989einbettungssatz}, \cite[Theorem 3]{gunther1991isometric})
Let $M$ be a smooth closed $n$-dimensional manifold, let $g$ be a smooth metric on $M$ and
let $F_0\in C^{\infty}(M,\mathbb{R}^q)$, $q\geq n(n+3)/2+5$, be a free embedding so that $g-F_0^{\ast}(\delta)$
is positive definite. Then, for each $\epsilon>0$, there exists a free embedding $F\in C^{\infty}(M,\mathbb{R}^q)$ so that 
\begin{equation*}
F^{\ast}(\delta)=g\hspace{0.5cm}\text{and}\hspace{0.5cm}\sup_{x\in M}{|F(x)-F_0(x)|}\leq \epsilon
\end{equation*}
\end{thm}

In order to prove this result, M. G\"unther decomposes the Riemannian metric $g-F_0^{\ast}(\delta)$ into a sum of tensor fields
that have a special shape and that are in particular compactly contained in a local chart (cf. \cite[Satz 2.2]{gunther1989einbettungssatz}).
Such a locally supported tensor field $h$ has the advantage that the metric $F_0^{\ast}(\delta)+h$ is induced by a free embedding $F$
which lies in an arbitrary small $C^0$-neighborhood of $F_0$ and that $F$ is just a local perturbation of $F_0$ (cf. \cite[Satz 2.4 and Section 6]{gunther1989einbettungssatz}).
This \textit{local modification method} is composed of \textit{two perturbation methods} where the \textit{first perturbation}
is based on the fact that the space dimension of the ambient space is greater than or equal to $n(n+3)/2+5$ (cf. \cite[Section 3 and 4]{gunther1989einbettungssatz}).
The \textit{first perturbation} (cf. \cite[Satz 4.1]{gunther1989einbettungssatz}) realizes the ``big part'' of the desired modification so that the remaining part
is arbitrary small in the $C^{2,\alpha}$-sense (cf. \cite[(6.1)]{gunther1989einbettungssatz}). The subsequent \textit{second perturbation} (cf. \cite[Satz 5.5]{gunther1989einbettungssatz}, \cite[Theorem]{gunther1989perturbation}) eliminates this small error. It is worth mentioning that the second dimension restriction in G\"unther's embedding theorem (cf. \cite[Satz 1.2]{gunther1989einbettungssatz}) is due to the fact
that it is always possible to construct a free embedding of a smooth closed manifold into the Euclidean space of dimension $n(n+5)/2$ 
(cf. \cite{nash1956imbedding}, \cite[Satz 1.2]{gunther1989einbettungssatz}, \cite[Theorem 2.1]{andrews2002notes}).

Within this framework we are interested in the following problem: Given a smooth closed manifold $M^n$ and a smooth one-parameter family of Riemannian metrics $(g(t))_{t\in[0,T]}$.

\textit{Is it possible to construct a family of mappings $F\in C^{\infty}(M\times [0,T^{\ast}],\mathbb{R}^{q(n)})$, where $T^{\ast}\in (0,T]$, so that for each $t\in [0,T^{\ast}]$
the mapping $F(\cdot,t)$ is an isometric embedding of the Riemannian manifold $(M,g(\cdot,t))$}?

We point out that the core of this problem is the regularity of the family of embeddings
with respect to the time parameter $t$.

In order to construct such a smooth family of isometric embeddings it seems sensible to refer
to the ``static situation'' at time $t=0$ and to analyze the time-dependency of one
of the perturbation techniques in \cite{nash1956imbedding} or \cite{gunther1989einbettungssatz}
for instance. In this work we examine the perturbation technique in \cite{gunther1989einbettungssatz} regarding
stability with respect to the time parameter. In particular we are interested in the
time-dependency of the
 part of G\"unther's modification technique which we have called \textit{second perturbation} (cf. \cite[Satz 5.5]{gunther1989einbettungssatz})
in the above description. Appropriate adjustments on the time-dependent situation yield the following
result:

\begin{theorem}\label{globeinb}
Let $M$ be a smooth closed $n$-dimensional manifold and let $(g(\cdot,t))_{t\in[0,T]}$ be a smooth family of Riemannian metrics on $M$. Given a free isometric
embedding $F_0\in C^{\infty}(M,\mathbb{R}^q)$ of the Riemannian manifold $(M,g(\cdot,0))$. Then, there exists $T^{\ast}\in (0,T]$
and a family of free embeddings $(F(\cdot,t))_{t\in[0,T^{\ast}]}\subset C^{\infty}(M,\mathbb{R}^{q})$ with $F\in C^{\infty}(M\times[0,T^{\ast}],\mathbb{R}^{q})$ and $F(\cdot,0)=F_0$ so that for all  $t\in [0,T^{\ast}]$ the equality
\begin{equation*}
F(\cdot,t)^{\ast}(\delta)=g(\cdot,t)\hspace{0.5cm}
\end{equation*}
holds on $M$.
\end{theorem}

In Section \ref{sol_per} we give a brief overview over M. G\"unther's method to solve \textit{the local perturbation problem}  (cf. \cite[Satz 5.5]{gunther1989einbettungssatz}). Subsequently, in Section \ref{stab_per}, we work out estimates to show that \textit{the local perturbation problem} 
is stable with respect to a time parameter. As a consequence we are able to construct a smooth family of mappings that
realizes the isometric embedding of a smooth family of Riemannian metrics  ``evolving'' in a small domain of a local chart, for a short time.
It stands to reason to apply this method to other charts iteratively but one needs to take into account that the first application
of this local modification technique generates a time-dependent family of free embeddings. Since some operators within this technique depend
on the ``initial'' free embedding (cf. \eqref{eq:5.18}) one needs to be careful about these iteration steps.

Finally, using the isometric embedding theorem in \cite[Folgerung 1.3]{gunther1989einbettungssatz} we obtain:

\begin{theorem}
Let $M$ be a smooth closed $n$-dimensional manifold and let $(g(\cdot,t))_{t\in[0,T]}$ be a smooth family of Riemannian metrics on $M$. Then, there exists $T^{\ast}\in (0,T]$
and a family of free embeddings $(F(\cdot,t))_{t\in[0,T^{\ast}]}\subset C^{\infty}(M,\mathbb{R}^{q(n)})$ where
 $q(n):=\max\left\{n(n+5)/2,(n+3)/2+5 \right\}$ with $F\in C^{\infty}(M\times[0,T^{\ast}],\mathbb{R}^{q(n)})$ so that for all $t\in [0,T^{\ast}]$ the 
equation
\begin{equation*}
F(\cdot,t)^{\ast}(\delta)=g(\cdot,t)\hspace{0.5cm}
\end{equation*}
holds on $M$.
\end{theorem}

An application of this result is the statement that \textit{geometric evolution equations} like
the \textit{Ricci flow on closed manifolds}, i.e.
\begin{equation*}
\begin{cases}
\frac{\partial}{\partial t}g(t)&=-2 Ric_{g(t)} \\
 g(0)&=g_0
\end{cases}
\end{equation*} 
where $g_0$ is a smooth metric on a smooth closed Riemannian manifold $M^n$, is in general ``visualisable'' 
in an Euclidean space of space dimension $N(n)$ where $N(n)$ has quadratic order with respect to the
dimension of the manifold $n$.

In this connection we want to mention the works \cite{rubinstein2005visualizing} and \cite{1311.0289}
which are concerned with the isometric embedding of the Ricci flow on surfaces of revolution.

Furthermore, we want to refer to the works \cite{andrews2002notes}, \cite{hong2006isometric}, \cite{de2016masterpieces} and \cite{gromov2017geometric} which address developments in the field of \textit{isometric embeddings of Riemannian manifolds}. But
we want to point out that each of these works has a special focus. It should also be mentioned that \cite[Chapter 1]{hong2006isometric}
contains further details of M. G\"unther's proof of the \textit{isometric embedding problem}. Worthy of particular mention is the work \cite{de2016masterpieces} in which the author
focuses attention on John Nash's contributions to mathematics. We also want to mention Deane Yang's work \cite{math/9807169} which is an
 ``informal expository note'' describing M. G\"unther's approach to the \textit{isometric embedding problem}.

This work is a part of the author's diploma thesis (\cite{1607.04259}), written under the supervision of Miles Simon at the Otto-von-Guericke-Universit\"at Magdeburg.

\section{G\"unther's solution to the local perturbation problem}\label{sol_per}

We consider the following \textit{local perturbation problem}: Given two open sets $U_1, U_2\subset B_1(0)$ where $\overline{U}_1\subset U_2$ and $\overline{U}_2 \subset B_1(0)$, a \textit{free immersion} $F_0\in C^{\infty}(\overline{B_1(0)},\mathbb{R}^q)$ (cf. Definition \ref{freemap}) as well as a \textit{perturbation} $f\in C^{\infty}(\overline{B_1(0)},\mathbb{R}^{\frac{n}{2}(n+1)})$ with $supp(f)\subset U_1$. 

\textit{Under which condition there exists a function $u\in C^{\infty}(\overline{B_1(0)},\mathbb{R}^q)$ with $supp(u)\subset U_2$ that solves the equation:
\begin{equation*}
\partial_i(F_0+u)\cdot \partial_j (F_0+u) =\partial_i F_0\cdot \partial_j F_0+f_{ij}
\end{equation*}
 for all $1\leq i\leq j\leq n$?}

We outline the method from \cite[Section 5]{gunther1989einbettungssatz} to show that such a ``modification function'' $u$ always exists, provided that $f$ is ``small enough'' in the $C^{2,\alpha}$-sense. 

This method is the foundation for the proof of Theorem \ref{globeinb} in Section \ref{stab_per}. 
The main purpose of this section is to recall the key ideas in \cite[Section 5]{gunther1989einbettungssatz}, especially
the \textit{formulation of the fixed-point problem} in Section \ref{sec:5.2},  
and to establish appropriate notation, which is adapted to the \textit{time-dependent situation} in Section \ref{stab_per} (cf. the ``fixed-point operator'' in \eqref{eq:5.18} for instance). 

As already mentioned, such called \textit{free immersions} play an important role in our context:
\begin{defi}\label{freemap}
Let $M$ be a smooth $n$-dimensional manifold. A mapping $F_0\in C^{\infty}(M,\mathbb{R}^q)$ is said to be a \textit{free immersion} if for every $x\in M$ and
every smooth chart $\varphi: U\longrightarrow V$ with $U\ni x$ the following set of $n(n+3)/2$ vectors
\begin{equation*}
\{\partial_i (F_0\circ \varphi^{-1})(\varphi(x))\}_{1\leq i \leq n}\cup \{\partial_i \partial_j(F_0\circ \varphi^{-1})(\varphi(x))\}_{1\leq i\leq j\leq n}
\end{equation*}
is linearly independent.
\end{defi}

In this context we refer to \cite[Section 1.2 and 1.3]{andrews2002notes} for a motivation for these sort of immersions.
It is worth mentioning that the $n$-dimensional sphere $\mathbb{S}^n$ has a free embedding into the Euclidean space of dimension $n(n+3)/2$ which
is obviously the smallest possible dimension (cf. \cite[Appendix 5]{gromov1970embeddings}).
In general, as already mentioned, 
each smooth closed Riemannian manifold possesses a smooth free embedding into the Euclidean space of dimension $n(n+5)/2$ (cf. \cite[Theorem 2.1]{andrews2002notes}).

The key idea in M. G\"unther's approach is to convert the \textit{local perturbation problem} into a \textit{fixed-point problem} of an appropriate operator. In
order to do this, he introduces some operators \cite[Section 5]{gunther1989einbettungssatz} which are ``parts'' of the ``fixed-point operator''
(cf. \eqref{eq:5.18}). We point out that we don't give detailed a priori arguments that motivate the choice of these operators, we just write down 
necessary information to prepare the arguments within the \textit{time-dependent situation} in Section \ref{stab_per}. Whenever any statement is proven in \cite[Section 5]{gunther1989einbettungssatz} we give a reference. There are some non-standard arguments that does not appear in \cite[Section 5]{gunther1989einbettungssatz}. We prove these statements 
as a part of the generalized, time-dependent situation in Section \ref{stab_per} and we give forward references to these arguments.



\subsection{Definition of auxiliary operators}
Throughout let $\alpha\in (0,1)$ be a fixed H\"older exponent. Let $a\in C^{\infty}_0(B_1(0))$ be fixed, then for all $1\leq i\leq n$  we define the operator (cf. \cite[(5.7)]{gunther1989einbettungssatz}):
\begin{align}\label{eq:5.2}
\begin{split}
N_i[a]: C^{2,\alpha}(B_1(0),\mathbb{R}^q)&\longrightarrow C^{0,\alpha}(B_1(0))\\
N_i[a](v)&:= 2\partial_i a \Delta v\cdot v+a\Delta v\cdot\partial_i v
\end{split}
\end{align}
If  $f\in C^{0,\alpha}(B_1(0))$ is a function, then $\Delta^{-1}f\in C^{2,\alpha}(B_1(0))$ is  \textit{the unique solution to Poisson's equation with trivial Dirichlet boundary data},
i.e
\begin{equation*}
\begin{cases}
\Delta u=f   &\text{on }B_1(0)\\
\hphantom{\Delta} u=0   &\text{on }\partial B_1(0)
\end{cases}
\end{equation*}

If $v\in C^{3,\alpha}(B_1(0),\mathbb{R}^q)$, then for all $1\leq i\leq j\leq n$ we define functions
$u^{(1)}_{ij}[a](v),u^{(2)}_{ij}[a](v)\in C^{2,\alpha}(B_1(0),\mathbb{R}^q)$ as follows
\begin{align*}
u^{(1)}_{ij}[a](v)&:= a\partial_i\Delta^{-1} N_j[a](v)+a\partial_j\Delta^{-1} N_i[a](v)+3\partial_i a \Delta^{-1} N_j[a](v)+3\partial_j a \Delta^{-1} N_i[a](v)\\
u^{(2)}_{ij}[a](v)&:= 4\partial_i a \partial_j a\,  v\cdot v+2a \partial_i a\, \partial_j v\cdot v+2a \partial_j a\, \partial_i v\cdot v+a^2\, \partial_i v\cdot\partial_j v
\end{align*} 
The following Lemmas are concerned with the formal computation of $\Delta u_{ij}^{(l)}[a](v)$:

\begin{lem}\label{lem:5.1}(cf. \cite[(5.8)]{gunther1989einbettungssatz})
Let $a\in C^{\infty}_0(B_1(0))$ be fixed and let $1\leq i\leq j\leq n$. Then there exists an operator:
\begin{equation*}
L_{ij}[a] : C^{2,\alpha}(B_1(0),\mathbb{R}^q)\longrightarrow C^{0,\alpha}(B_1(0))
\end{equation*}
that has the shape:
\begin{align*}
L_{ij}[a](v)=&\sum_{l\in\{i,j\}}{\sum_{\substack{s_1,s_2\in\mathbb{N}^n\\ |s_1|+|s_2|=3,\ |s_2|\leq 2}}{C_{ij}(s_1,s_2)\cdot\partial^{s_1} a\cdot\partial^{s_2}}(\Delta^{-1}N_l[a](v))}\\
\end{align*}
so that for all $v\in C^{3,\alpha}(B_1(0),\mathbb{R}^q)$ the equation
\begin{align*}
&\Delta u_{ij}^{(1)}[a](v)=a\,\partial_i N_j[a](v)+a \,\partial_j N_i[a](v)-L_{ij}[a](v)
\end{align*}
is satisfied.
\end{lem}


\begin{lem}\label{lem:5.2}(cf. \cite[(5.9)]{gunther1989einbettungssatz})
Let $a\in C^{\infty}_0(B_1(0))$ be fixed and let $1\leq i\leq j\leq n$. Then there exists an operator:
\begin{equation*}
R_{ij}[a] : C^{2,\alpha}(B_1(0),\mathbb{R}^q)\longrightarrow C^{0,\alpha}(B_1(0))
\end{equation*}
that has the shape
\begin{equation*}
R_{ij}[a](v)=\sum_{\substack{s_1,s_2,s_3,s_4\in\mathbb{N}^n\\|s_1|+|s_2|+|s_3|+|s_4|=4,\ |s_3|,|s_4|\leq 2}}{C_{ij}(s_1,s_2,s_3,s_4)\cdot\partial^{s_1} a\ \partial^{s_2} a\ \partial^{s_3} v \cdot \partial^{s_4} v}
\end{equation*}
so that for all $v\in C^{3,\alpha}(B_1(0),\mathbb{R}^q)$ the equation:
\begin{align*}
\Delta u^{(2)}_{ij}[a](v)=a\partial_i N_j[a](v)+a\partial_j N_i[a](v)+R_{ij}[a](v)
\end{align*}
is satisfied.
\end{lem} 

Now, for all $1\leq i\leq j\leq n$ we define: 
\begin{equation*}
u_{ij}[a](v):= u_{ij}^{(2)}[a](v)-u_{ij}^{(1)}[a](v)
\end{equation*}
Then, using Lemma \ref{lem:5.1} und Lemma \ref{lem:5.2} we obtain
\begin{align*}
\begin{cases}
\Delta u_{ij}[a](v)=M_{ij}[a](v) &\text{on }B_1(0)\\
\hphantom{\Delta} u_{ij}[a](v)=0 &\text{on }\partial B_1(0)
\end{cases}
\end{align*}
where
\begin{align}\label{eq:5.11}
\begin{split}
M_{ij} : C^{2,\alpha}(B_1(0),\mathbb{R}^q)&\longrightarrow C^{0,\alpha}(B_1(0))\\
M_{ij}[a](v)&:= L_{ij}[a](v)+R_{ij}[a](v)
\end{split}
\end{align}
The following Lemma contains the \textit{crucial idea} in order to solve the \textit{local perturbation problem}:

\begin{lem}\label{lem:5.3}(cf. \cite[Lemma 5.3]{gunther1989einbettungssatz})
Suppose that $F_0 \in C^{2,\alpha}(B_1(0),\mathbb{R}^q)$, $a\in C^{\infty}_0(B_1(0)),$  $v \in C^{3,\alpha}(B_1(0),\mathbb{R}^q)$ and $f \in C^{2,\alpha}(B_1(0),\mathbb{R}^{\frac{n}{2}(n+1)})$ satisfy
\begin{align}
\label{eq:5.7}
\partial_i F_0\cdot v&=-a\Delta^{-1}N_i[a](v)  & &\text{if } 1\leq i\leq n\\
\label{eq:5.8}
\partial_i \partial_j F_0\cdot v&=-\frac{1}{2}f_{ij}+\frac{1}{2}\Delta^{-1}M_{ij}[a](v) & &\text{if } 1\leq i\leq  j\leq n
\end{align}
on $B_1(0)$. Then the function $F:= F_0+a^2 v$ satisfies
\begin{equation}\label{eq:5.17}
\partial_i F\cdot \partial_j F=\partial_i F_0\cdot \partial_j F_0+a^2 f_{ij}
\end{equation}
for all $1\leq i\leq j\leq n$ on $B_1(0)$.
\end{lem}

\subsection{Formulation of the fixed-point problem}
\label{sec:5.2}
In order to solve the \textit{local perturbation problem} we examine whether the system of equations in Lemma \ref{lem:5.3} is solvable.
As already mentioned, it is our aim to convert this system into a \textit{fixed-point equation}. In order to do this, we introduce some operators.

Let $F_0\in C^{\infty}(\overline{B_1(0)},\mathbb{R}^q)$ be a free immersion then $A[F_0]\in C^{\infty}(\overline{B_1(0)},\mathbb{R}^{\frac{n}{2}(n+3)\times q})$ is defined to be equal to the matrix valued function whose first $n$ rows are the functions $\partial_i F_0^{\top}$ where $1\leq i\leq n$ and
the remaining rows are the functions $\partial_i\partial_j F_0^{\top}$ where $1\leq i\leq j\leq n$
in lexicographic order.

If $\{h_i \}_{1\leq i\leq n}\subset C(\overline{B_1(0)})$ and $\{f_{ij} \}_{1\leq i\leq j\leq n}\subset C(\overline{B_1(0)})$ are
sets of functions, the functions $h\in C(\overline{B_1(0)},\mathbb{R}^n)$ and $f\in C(\overline{B_1(0)},\mathbb{R}^{\frac{n}{2}(n+1)})$ are
the (column-) vector valued functions that have the analogous functions as component functions. Again, we are using the lexicographic order 
if we have two indices. Furthermore, the function $\Theta[F_0]\in C^{\infty}(\overline{B_1(0)},\mathbb{R}^{q\times \frac{n}{2}(n+3)})$ is defined as
\begin{align}\label{eq:5.41}
\Theta[F_0](x)=A^{\top}[F_0](x)\cdot( A[F_0](x)  A^{\top}[F_0](x))^{-1}
\end{align}
so that the linear operator
\begin{align}\label{eq:5.19}
\begin{split}
E[F_0]:  C^{m,\alpha}(B_1(0),\mathbb{R}^n) &\times C^{m,\alpha}(B_1(0),\mathbb{R}^{\frac{n}{2}(n+1)})\longrightarrow C^{m,\alpha}(B_1(0),\mathbb{R}^q) \\  
E[F_0](h,f)(x)&:=\Theta[F_0](x)\cdot\begin{pmatrix} h(x)\\ f(x) \end{pmatrix}
\end{split}
\end{align}
is well-defined. By definition we have for all $x\in B_1(0)$ the equality
\begin{equation*}
 A[F_0](x)\cdot \Theta[F_0](x)=I
\end{equation*}
where $I\in\mathbb{R}^{\frac{n}{2}(n+3)\times\frac{n}{2}(n+3)}$ is the identity matrix. This equality implies
\begin{equation*}
  \begin{alignedat}{2}
    \partial_i F_0(x)\cdot E[F_0](h,f)(x)
    &=h_i(x) 
    && \hspace{1cm}\text{if }1\leq i\leq n\\
		\partial_i\partial_j F_0(x)\cdot E[F_0](h,f)(x)
    & =f_{ij}(x)
    && \hspace{1cm}\text{if }1\leq i\leq j\leq n
  \end{alignedat}
\end{equation*}
Keeping equations $\eqref{eq:5.7}$ and $\eqref{eq:5.8}$ in mind it is reasonable to introduce the following operators
\begin{align}
\begin{split}\label{eq:5.14}
 P[a]: C^{m,\alpha}(B_1(0),\mathbb{R}^q)&\longrightarrow C^{m,\alpha}(B_1(0),\mathbb{R}^n)  \\
(P_i[a](v))_{1\leq i\leq n}&:=(a\Delta^{-1} N_i[a](v))_{1\leq i\leq n}
\end{split}
\intertext{and}
\begin{split}\label{eq:5.15}
 Q[a]: C^{m,\alpha}(B_1(0),\mathbb{R}^q)&\longrightarrow C^{m,\alpha}(B_1(0),\mathbb{R}^{\frac{n}{2}(n+1)})\\
(Q_{ij}[a](v))_{1\leq i\leq j\leq n}&:=(\Delta^{-1} M_{ij}[a](v))_{1\leq i\leq j\leq n}
\end{split}
\end{align}
where $a\in C^{\infty}_0(B_1(0))$, $m\geq 2$ and the operators $N_i$ and $M_{ij}$ are defined in \eqref{eq:5.2} and \eqref{eq:5.11}. Finally we define the operator:
\begin{align}\label{eq:5.18}
\begin{split}
\Phi[F_0,a,f]: C^{m,\alpha}(B_1(0),\mathbb{R}^q)&\longrightarrow C^{m,\alpha}(B_1(0),\mathbb{R}^q)\\
\Phi[F_0,a,f](v)&:=-E[F_0]\left(P[a](v), \frac{1}{2}f- \frac{1}{2} Q[a](v)\right)
\end{split}
\end{align}
where $f\in C^{m,\alpha}(B_1(0),\mathbb{R}^{\frac{n}{2}(n+1)})$ and $m\geq 2$. Hence, the system in Lemma \ref{lem:5.3} is solved if the \textit{fixed-point equation}
\begin{equation}\label{eq:5.16}
v=\Phi[F_0,a,f](v)
\end{equation}
is solved.

As mentioned at the beginning of this section, we assume that $supp(f)\subset U_1$. Hence, we may choose a specific function
$a\in C^{\infty}_0(B_1(0))$ so that $f=a^2 f$. Then, the solution to the \textit{local perturbation problem} follows from \eqref{eq:5.17}.

It remains show that the operator in \eqref{eq:5.18} has a fixed-point. In particular, we are interested in an explicit method to construct such a fixed-point. In order to prove the existence,
we write down appropriate $C^{m,\alpha}$-estimates in Section \ref{sec:5.4}. By definition of the operators $P[a]$ and $Q[a]$, we notice that the Dirichlet boundary value problem for Poisson's equation is important here. In the following section we state some standard facts within the context of \textit{the theory
of elliptic differential equations}. 


\subsection{Dirichlet boundary value problem for Poisson's equation}

Let $f\in C^{0,\alpha}(B_1(0))$ and $\varphi\in C^{2,\alpha}(B_1(0))$ then there exists a unique $u\in C^{2,\alpha}(B_1(0))$ satisfying
\begin{align}\label{eq:5.21}
\begin{cases}
\Delta u=f &\text{on }B_1(0)\\
\hphantom{\Delta} u=\varphi &\text{on }\partial B_1(0)
\end{cases}
\end{align}
And we have the following estimate:
\begin{equation}\label{eq:5.20}
\left|u \right|_{C^{2,\alpha}}\leq C(\alpha)\cdot \left( \left|f \right|_{C^{0,\alpha}} + \left|\varphi \right|_{C^{2,\alpha}}\right)
\end{equation}
For our considerations, in particular for \textit{regularity analysis}, the following \textit{higher-order estimates} are important:
\begin{lem}(cf. \cite[(5.4)]{gunther1989einbettungssatz}) 
Given a function $f\in C^{m,\alpha}(B_1(0))$, then the unique solution $u\in {C^{m+2,\alpha}}(B_1(0))$ of \eqref{eq:5.21} with trivial boundary data satisfies
\begin{equation}\label{eq:5.22}
\left|u \right|_{C^{m+2,\alpha}}\leq C(m,\alpha)\cdot \left|f \right|_{C^{m,\alpha}}
\end{equation}
\end{lem}
If $supp(f)\subset B_1(0)$ then the estimate \eqref{eq:5.22} may be slightly improved:
\begin{lem}(cf. \cite[Lemma 5.1]{gunther1989einbettungssatz})
Let $f\in C^{m,\alpha}(B_1(0))$ where $m\geq 1$ and $supp(f)\subset B_R(0)$ then, the unique solution
$u\in C^{m+2,\alpha}(B_1(0))$ of \eqref{eq:5.21} with trivial boundary data satisfies:
\begin{equation}\label{eq:5.30}
\left|u \right|_{C^{m+2,\alpha}}\leq K(\alpha,R)\cdot \left|f \right|_{C^{m,\alpha}}+C(m,\alpha,R)\cdot \left|f \right|_{C^{m-1,\alpha}}
\end{equation} 
\end{lem}


\subsection{H\"older-estimates of the auxiliary operators}
\label{sec:5.4}

Using the elliptic estimates from the previous section one can show the following continuity- and regularity estimates for
the operators $N_i[a]$ and $M_{ij}[a]$. 
We mention that the proofs of Lemma \ref{lem:7.2} and Lemma \ref{lem:7.4} in Section \ref{stab_per} contain proofs of the following estimates:

\begin{lem}(cf. \cite[Lemma 5.4]{gunther1989einbettungssatz}) 
Let $a\in C_0^{\infty}(B_1(0))$ be fixed, then we have the following estimates: If $v_1,v_2\in C^{2,\alpha}(B_1(0),\mathbb{R}^q)$ then
\begin{equation}\label{eq:5.31}
\left|N_i[a](v_1)-N_i[a](v_2) \right|_{C^{0,\alpha}} \leq K(\alpha,a)\cdot \left(\left|v_1 \right|_{C^{2,\alpha}}+\left|v_2 \right|_{C^{2,\alpha}} \right)\left|v_1-v_2 \right|_{C^{2,\alpha}}
\end{equation}
for all $1\leq i \leq n$ and
\begin{align}\label{eq:5.33}
\left|M_{ij}[a](v_1) -M_{ij}[a](v_2) \right|_{C^{0,\alpha}}\leq K(\alpha,a)\cdot \left(\left|v_1 \right|_{C^{2,\alpha}}+\left|v_2 \right|_{C^{2,\alpha}} \right)\left|v_1-v_2 \right|_{C^{2,\alpha}}
\end{align}
for all $1\leq i \leq j\leq n$. If $m\in\mathbb{N}$ and $v\in {C^{m+2,\alpha}(B_1(0),\mathbb{R}^q)}$ then
\begin{equation*}\label{eq:5.32}
\left|N_i[a](v) \right|_{C^{m,\alpha}}+\left|M_{ij}[a](v) \right|_{C^{m,\alpha}}\leq K(\alpha,a)\cdot\left|v \right|_{C^{m+2,\alpha}}\left|v \right|_{C^{2,\alpha}}+C(m,\alpha,a)\cdot\left|v \right|^2_{C^{m+1,\alpha}}
\end{equation*}
for all $1\leq i \leq j \leq n$.
\end{lem}

The following continuity- and regularity estimates for the operators $N_i[a]$ and $M_{ij}[a]$ are proved in Lemma \ref{lem:7.5} within the
time-dependent context:

\begin{lem}(cf. \cite[(5.17)/(5.18)]{gunther1989einbettungssatz})\label{lem:5.d}
Let $a\in C_0^{\infty}(B_1(0))$ be fixed, then the operators $P$ (cf. \eqref{eq:5.14}) and $Q$ (cf. \eqref{eq:5.15})
are satisfying the following estimates: For all $v_1,v_2\in C^{2,\alpha}(B_1(0),\mathbb{R}^q)$:
\begin{align}\label{eq:5.37}
\begin{split}
&\left|P[a](v_1)-P[a](v_2) \right|_{C^{2,\alpha}}+\left|Q[a](v_1)-Q[a](v_2) \right|_{C^{2,\alpha}}\\
&\hspace{3cm}\leq K(\alpha,a)\cdot\left(\left|v_1 \right|_{C^{2,\alpha}} +\left|v_2 \right|_{C^{2,\alpha}}\right)\cdot \left|v_1-v_2 \right|_{C^{2,\alpha}}
\end{split}
\end{align}
If $m\geq 2$ and $v\in {C^{m,\alpha}(B_1(0),\mathbb{R}^q)}$ then
\begin{align}\label{eq:5.38}
\left|P[a](v) \right|_{C^{m,\alpha}}+\left|Q[a](v) \right|_{C^{m,\alpha}}\leq K(\alpha,a)\cdot\left|v \right|_{C^{m,\alpha}} \left|v \right|_{C^{2,\alpha}}+C(m,\alpha,a)\cdot \left|v \right|_{C^{m-1,\alpha}}^2
\end{align}
\end{lem}

The next two results show important properties of the operator $E[F_0]$ (cf. \eqref{eq:5.19}). We refer to 
the proofs of Lemma \ref{lem:7.6} and Lemma \ref{lem:7.7}.
\begin{lem}\label{lem:5.a}
Let $F_0\in C^{\infty}(\overline{B_1(0)},\mathbb{R}^q)$ be a free immersion. Then the linear operator $E[F_0]$ is continuous and for all
 $h\in C^{m,\alpha}(B_1(0),\mathbb{R}^n)$ and $f\in C^{m,\alpha}(B_1(0),\mathbb{R}^{\frac{n}{2}(n+1)})$ we have:
\begin{equation}\label{eq:5.46}
\left|E[F_0](h,f) \right|_{C^{m,\alpha}}\leq C(m,\alpha,F_0)\cdot\left(\left|h \right|_{C^{m,\alpha}}+\left|f \right|_{C^{m,\alpha}}\right)
\end{equation}
\end{lem}

Throughout let $\left\Vert E[F_0] \right\Vert_{m,\alpha}$ be  the \textit{operator norm} of $E[F_0]$. The estimate \eqref{eq:5.46} may be improved:

\begin{lem}(cf. \cite[Lemma 5.2]{gunther1989einbettungssatz})\label{lem:5.e}
Let $m\geq 3$, then we have the following estimate:
\begin{align}\label{eq:5.49}
\begin{split}
\left|E[F_0](h,f) \right|_{C^{m,\alpha}}\leq &\left\Vert E[F_0] \right\Vert_{2,\alpha}\cdot \left(\left|h \right|_{C^{m,\alpha}}+\left|f \right|_{C^{m,\alpha}}\right)+C(m,\alpha,F_0)\cdot\left(\left|h \right|_{C^{m-1,\alpha}}+\left|f \right|_{C^{m-1,\alpha}}\right)
\end{split}
\end{align}
for all $h\in C^{m,\alpha}(B_1(0),\mathbb{R}^n)$ and $f\in C^{m,\alpha}(B_1(0),\mathbb{R}^{\frac{n}{2}(n+1)})$.
\end{lem}

\subsection{Solution to the fixed-point problem}
Now we show that the fixed-point equation in \eqref{eq:5.16} has a solution if the quantity $\left\Vert E[F_0]\right\Vert_{2,\alpha} \left|E[F_0](0,f) \right|_{C^{2,\alpha}}$
 is sufficiently small. In order to show this statement, we construct an \textit{approximation sequence} via a \textit{fixed-point iteration} (cf. \eqref{eq:5.47}).
\begin{lem}\label{lem:5.13}
Let $a\in C^{\infty}_0(B_1(0))$ be fixed, then there exists $\vartheta(\alpha,a)>0$ satisfying the following property:
If $F_0\in C^{\infty}(\overline{B_1(0)},\mathbb{R}^q)$ is a free immersion and $f\in C^{\infty}(\overline{B_1(0)},\mathbb{R}^{\frac{n}{2}(n+1)})$ such that
\begin{equation*}
\left\Vert E[F_0]\right\Vert_{2,\alpha}\left|E[F_0](0,f) \right|_{C^{2,\alpha}}\leq \vartheta
\end{equation*}
then, the operator $\Phi[F_0,a,f]$ (cf. \eqref{eq:5.18}) has a fixed-point $v\in C^{\infty}(\overline{B_1(0)},\mathbb{R}^q)$. This fixed-point 
satisfies the following estimate:
\begin{equation*}
\left|v \right|_{C^{2,\alpha}}\leq \left|E[F_0](0,f) \right|_{C^{2,\alpha}}
\end{equation*}
\end{lem}

\begin{proof}
Let $(v_k)_{k\in\mathbb{N}}\subset C^{\infty}(\overline{B_1(0)},\mathbb{R}^q)$ defined as follows:
\begin{align}\label{eq:5.47}
v_k=
\begin{cases}
0 & \text{if }  k=0\\
\Phi[F_0,a,f](v_{k-1})  &\text{if } k\geq 1
\end{cases}
\end{align}
We show that this sequence is a $C^{2,\alpha}$-Cauchy sequence if $\left\Vert E[F_0]\right\Vert_{2,\alpha} \left|E[F_0](0,f) \right|_{C^{2,\alpha}} $
is small enough.
In order to prove this fact, we derive an upper bound at first.
\begin{align*}
\left|v_k \right|_{C^{2,\alpha}}\stackrel{\eqref{eq:5.47}}{=}&\left|\Phi[F_0,a,f](v_{k-1})\right|_{C^{2,\alpha}}\\
\stackrel{\eqref{eq:5.18}}{\leq}&\left|E[F_0]\left(P[a](v_{k-1}),-\frac{1}{2}Q[a](v_{k-1}) \right)\right|_{C^{2,\alpha}}+\left|E[F_0]\left(0,\frac{1}{2}f \right)\right|_{C^{2,\alpha}}\\
\stackrel{\eqref{eq:5.46}}{\leq}& \left\Vert E[F_0] \right\Vert_{2,\alpha}\cdot\left(\left|P[a](v_{k-1}) \right|_{C^{2,\alpha}} +\left|Q[a](v_{k-1}) \right|_{C^{2,\alpha}}\right)+\frac{1}{2}\left|E[F_0]\left(0,f \right)\right|_{C^{2,\alpha}}\\
\stackrel{\eqref{eq:5.38}}{\leq}&K_1(\alpha,a)\cdot \left\Vert E[F_0] \right\Vert_{2,\alpha}\cdot\left|v_{k-1} \right|_{C^{2,\alpha}}^2+\frac{1}{2}\left|E[F_0]\left(0,f \right)\right|_{C^{2,\alpha}}
\end{align*}
If
\begin{equation*}
\left\Vert E[F_0]\right\Vert_{2,\alpha} \left|E[F_0](0,f) \right|_{C^{2,\alpha}}\leq \frac{1}{2K_1(\alpha,a)}
\end{equation*}
then
\begin{align*}
\left|v_k \right|_{C^{2,\alpha}}\leq \frac{1}{2}\left(\frac{\left|v_{k-1} \right|_{C^{2,\alpha}}^2}{ \left|E[F_0](0,f) \right|_{C^{2,\alpha}}}+ \left|E[F_0]\left(0,f \right)\right|_{C^{2,\alpha}}\right)
\end{align*}
Since $v_0=0$, using Lemma \ref{lem:7.8}, we obtain the estimate
\begin{equation}\label{eq:5.48}
\left|v_k \right|_{C^{2,\alpha}}\leq \left|E[F_0]\left(0,f \right)\right|_{C^{2,\alpha}}
\end{equation}
for all $k\in\mathbb{N}$ and consequently
\begin{align*}
\left|v_{k+1}-v_{k} \right|_{C^{2,\alpha}}\stackrel{\eqref{eq:5.47}}{=}&\left|\Phi[F_0,a,f](v_{k})-\Phi[F_0,a,f](v_{k-1}) \right|_{C^{2,\alpha}}\\
\stackrel{\eqref{eq:5.18}}{=}&\left|E[F_0]\left(P[a](v_k)-P[a](v_{k-1}),-\frac{1}{2}(Q[a](v_k)-Q[a](v_{k-1}))\right) \right|_{C^{2,\alpha}}\\
\stackrel{\eqref{eq:5.46}}{\leq}& \left\Vert E[F_0]\right\Vert_{2,\alpha}\cdot\Bigl(\left|P[a](v_k)-P[a](v_{k-1})\right|_{C^{2,\alpha}}+\left| Q[a](v_k)-Q[a](v_{k-1}) \right|_{{C^{2,\alpha}}}\Bigr)\\
\stackrel{\eqref{eq:5.37}}{\leq}& \frac{1}{2} K_2(\alpha,a)\cdot\left\Vert E[F_0]\right\Vert_{2,\alpha}\cdot\left(\left|v_k \right|_{C^{2,\alpha}} +\left|v_{k-1} \right|_{C^{2,\alpha}}\right) \left|v_k-v_{k-1} \right|_{C^{2,\alpha}}\\
\stackrel{\eqref{eq:5.48}}{\leq}& K_2(\alpha,a)\cdot \left\Vert E[F_0]\right\Vert_{2,\alpha}\left|E[F_0](0,f)\right|_{C^{2,\alpha}} \left|v_k-v_{k-1}\right|_{C^{2,\alpha}}\\
\end{align*}
If
\begin{equation*}
\left\Vert E[F_0]\right\Vert_{2,\alpha} \left|E[F_0](0,f)\right|_{C^{2,\alpha}} \leq \frac{1}{2K_2(\alpha,a)}
\end{equation*}
then
\begin{equation*}
\left|v_{k+1}-v_{k} \right|_{C^{2,\alpha}}\leq \frac{1}{2} \left|v_k-v_{k-1}\right|_{C^{2,\alpha}}
\end{equation*}
which shows that $(v_k)_{k\in\mathbb{N}}\subset C^{\infty}(\overline{B_1(0)},\mathbb{R}^q)$ is a $C^{2,\alpha}$-Cauchy sequence with boundary element
 $v\in C^{2,\alpha}(B_1(0),\mathbb{R}^q)$. In order to prove that $v\in C^{\infty}(\overline{B_1(0)},\mathbb{R}^q)$ holds, we show
that the sequence  $(v_k)_{k\in\mathbb{N}}$ is bounded in $\left|\cdot \right|_{C^{m,\alpha}}$. Let $m\geq 3$, suppose that
\begin{equation}\label{eq:5.50}
\left|v_{k} \right|_{C^{m-1,\alpha}}\leq \eta
\end{equation}
for all $k\in\mathbb{N}$. Then
\begin{align*}
&\left|v_k \right|_{C^{m,\alpha}}\\
\stackrel{\eqref{eq:5.47}}{=}&\left|\Phi[F_0,a,f](v_{k-1}) \right|_{C^{m,\alpha}}\\
\stackrel{\eqref{eq:5.18}}{\leq}&\left|E[F_0]\left(P[a](v_{k-1}),-\frac{1}{2}Q[a](v_{k-1}) \right)\right|_{C^{m,\alpha}}+\left|E[F_0]\left(0,\frac{1}{2}f \right)\right|_{C^{m,\alpha}}\\
\stackrel{\eqref{eq:5.49}}{\leq}&\left\Vert E[F_0] \right\Vert_{2,\alpha}\cdot\left(\left|P[a](v_{k-1}) \right|_{C^{m,\alpha}}+\left|Q[a](v_{k-1}) \right|_{C^{m,\alpha}} \right)\\
&+C_1(m,\alpha,F_0)\cdot \left(\left|P[a](v_{k-1}) \right|_{C^{m-1,\alpha}}+\left|Q[a](v_{k-1}) \right|_{C^{m-1,\alpha}} \right)+\frac{1}{2}\left|E[F_0]\left(0,f \right)\right|_{C^{m,\alpha}}\\
\stackrel{\eqref{eq:5.38}}{\leq}&K_3(\alpha,a) \cdot\left\Vert E[F_0] \right\Vert_{2,\alpha}\cdot\left|v_{k-1} \right|_{C^{m,\alpha}}\left|v_{k-1} \right|_{C^{2,\alpha}}+C_2(m,\alpha,F_0,a)\cdot \left|v_{k-1} \right|_{C^{m-1,\alpha}}^2\\
&\hspace{2cm}+\frac{1}{2}\left|E[F_0]\left(0,f \right)\right|_{C^{m,\alpha}}\\
\stackrel{\eqref{eq:5.48}}{\leq}&K_3(\alpha,a) \cdot\left\Vert E[F_0] \right\Vert_{2,\alpha}\left|E[F_0](0,f) \right|_{C^{2,\alpha}} \cdot\left|v_{k-1} \right|_{C^{m,\alpha}}+C_2(m,\alpha,F_0,a)\cdot \left|v_{k-1} \right|_{C^{m-1,\alpha}}^2\\
&\hspace{2cm}+\frac{1}{2}\left|E[F_0]\left(0,f \right)\right|_{C^{m,\alpha}}
\end{align*}
If
\begin{equation*}
\left\Vert E[F_0] \right\Vert_{2,\alpha}\cdot\left|E[F_0](0,f) \right|_{C^{2,\alpha}} \leq \frac{1}{2K_3(\alpha,a)}
\end{equation*}
then
\begin{align*}
\left|v_k \right|_{C^{m,\alpha}}\stackrel{\eqref{eq:5.50}}{\leq}& \frac{1}{2}\left(\left|v_{k-1} \right|_{C^{m,\alpha}}+\left|E[F_0]\left(0,f \right)\right|_{C^{m,\alpha}}+C(m,\alpha,F_0,a,\eta)\right)
\end{align*}
Since $v_0=0$, Lemma \ref{lem:7.8} implies
\begin{equation*}
\left|v_k \right|_{C^{m,\alpha}}\leq \left|E[F_0]\left(0,f \right)\right|_{C^{m,\alpha}}+C(m,\alpha,F_0,a,\eta)
\end{equation*}
for all $k\in\mathbb{N}$.
\end{proof}
Since $supp(f)\subset U_1$ we have $f= a^2 f$ for an appropriate choice of $a\in C^{\infty}_0(B_1(0))$. We conclude:
\begin{theorem}(cf. \cite[Satz 5.5]{gunther1989einbettungssatz})\label{satz:5.1}
Let $U_1, U_2\subset B_1(0)\subset \mathbb{R}^n$ be open sets satisfying $\overline{U}_1\subset U_2$ and $\overline{U}_2 \subset B_1(0)$. Then, there exist constants
$\vartheta,C>0$ that have the following property: If $F_0\in C^{\infty}(\overline{B_1(0)},\mathbb{R}^q)$ is a free immersion and $f\in C^{\infty}(\overline{B_1(0)},\mathbb{R}^{\frac{n}{2}(n+1)})$ with $supp(f)\subset U_1$ such that
\begin{equation*}
\left\Vert E[F_0]\right\Vert_{2,\alpha}\left|E[F_0](0,f) \right|_{C^{2,\alpha}}\leq \vartheta
\end{equation*}
then, there exists $u\in C^{\infty}(\overline{B_1(0)},\mathbb{R}^q)$ with $supp(u)\subset U_2$ such that for all $1\leq i\leq j\leq n$ the equation
\begin{equation*}
\partial_i(F_0+u)\cdot \partial_j (F_0+u) =\partial_i F_0\cdot \partial_j F_0+f_{ij}
\end{equation*}
is satisfied. Furthermore, the estimate
\begin{equation*}
\left|u \right|_{C^{2,\alpha}}\leq C \cdot \left|E[F_0](0,f) \right|_{C^{2,\alpha}}
\end{equation*}
holds. 
\end{theorem}

\section{Stability of the solution to the local perturbation problem}\label{stab_per}
In this section we show that M. G\"unther's technique can be used to prove that, around a \textit{free isometric embedding}, it is possible to construct a smooth
parametrized family of free embeddings that realizes an isometric embedding of a given smooth one-parameter family of Riemannian metrics, for a short time. This is the content of Theorem \ref{globeinb}. To prove this statement, we start from the local viewpoint. In particular we want to establish the maximal regularity
in the local context. Subsequently, we apply this local argument to each local chart of a given finite covering of the underlying manifold.
But we need to pay attention to the fact that the \textit{first time-dependent perturbation} of the initial embedding has generated a \textit{family of free embeddings},  so that
the ``fixed-point operator'' in \eqref{eq:5.18} is \textit{time-dependent} too.

\subsection{Construction of a local solution}\label{sec:7.1}
First of all we are interested in the \textit{$C^0$-stability} of a given free embedding in the local context.
The following Lemma states that it is possible to choose the constant $\vartheta>0$ in Lemma \ref{lem:5.13} small enough so
that two fixed-points of the operator $\Phi$  satisfy a specific \textit{continuity estimate} if the corresponding energies are ``small enough''
in the $C^{2,\alpha}$-sense. Throughout $\alpha\in (0,1)$ be fixed.
\begin{lem}\label{lem:7.1}
Let $a\in C^{\infty}_0(B_1(0))$ be fixed. There exists a constant $\vartheta(\alpha,a)>0$ satisfying the following property: If $F_0\in C^{\infty}(\overline{B_1(0)},\mathbb{R}^q)$ is a free immersion and $f^{(1)},f^{(2)}\in C^{\infty}(\overline{B_1(0)},\mathbb{R}^{\frac{n}{2}(n+1)})$ so that
\begin{align}\label{eq:7.3}
\begin{split}
\left\Vert E[F_0]\right\Vert_{2,\alpha}\cdot\left(\left|E[F_0](0,f^{(1)}) \right|_{C^{2,\alpha}}+\left|E[F_0](0,f^{(2)}) \right|_{C^{2,\alpha}} \right)&\leq \vartheta
\end{split}
\end{align}
then, the mappings $v^{(1)},v^{(2)}\in C^{\infty}(\overline{B_1(0)},\mathbb{R}^q)$ from Lemma \ref{lem:5.13} which solve
\begin{align*}
v^{(1)}=\Phi[F_0,a,f^{(1)}](v^{(1)})\hspace{0.25cm}\text{and}\hspace{0.25cm}v^{(2)}=\Phi[F_0,a,f^{(2)}](v^{(2)})
\end{align*}
satisfy the estimate
\begin{equation}\label{eq:7.5}
\left|v^{(1)}-v^{(2)}\right|_{C^{2,\alpha}}\leq \left|E[F_0](0,f^{(1)}-f^{(2)}) \right|_{C^{2,\alpha}}
\end{equation}
\end{lem}

\begin{proof}
Let $(v^{(1)}_k)_{k\in\mathbb{N}}, (v^{(2)}_k)_{k\in\mathbb{N}}\subset C^{\infty}(\overline{B_1(0)},\mathbb{R}^q)$ be the approximations from \eqref{eq:5.47}, i.e.:
\begin{align*}
v^{(j)}_k=
\begin{cases}
0 & \text{if }  k=0\\
\Phi[F_0,a,f^{(j)}](v^{(1)}_{k-1})  &\text{if } k\geq 1
\end{cases}
\end{align*}
Then for all $k\geq 1$ we have
\begin{align*}
&\left|v^{(1)}_k-v^{(2)}_k\right|_{C^{2,\alpha}}=\left|\Phi[F_0,a,f^{(1)}](v^{(1)}_{k-1})-\Phi[F_0,a,f^{(2)}](v^{(2)}_{k-1})\right|_{C^{2,\alpha}}\\
\stackrel{\eqref{eq:5.18}}{\leq} & \left\Vert E[F_0]\right\Vert_{2,\alpha}\cdot  \left|P[a](v^{(1)}_{k-1})-P[a](v^{(1)}_{k-1}) \right|_{C^{2,\alpha}}+\left\Vert E[F_0]\right\Vert_{2,\alpha}\cdot \left|Q[a](v^{(1)}_{k-1})-Q[a](v^{(2)}_{k-1}) \right|_{C^{2,\alpha}}\\
&\hspace{3cm}+\frac{1}{2}  \left|E[F_0](0,f^{(1)}-f^{(2)})\right|_{C^{2,\alpha}}\\
\stackrel{\eqref{eq:5.37}}{\leq} & K(\alpha,a)\cdot\left\Vert E[F_0]\right\Vert_{2,\alpha}\cdot \left(|v^{(1)}_{k-1} |_{C^{2,\alpha}}+|v^{(2)}_{k-1} |_{C^{2,\alpha}}\right)\cdot|v^{(1)}_{k-1}-v^{(2)}_{k-1} |_{C^{2,\alpha}}\\
&\hspace{3cm}+\frac{1}{2}  \left|E[F_0](0,f^{(1)}-f^{(2)})\right|_{C^{2,\alpha}}\\
\stackrel{\eqref{eq:5.48}}{\leq} & K(\alpha,a)\cdot\left\Vert E[F_0]\right\Vert_{2,\alpha} \left(\left|E[F_0](0,f^{(1)}) \right|_{C^{2,\alpha}}+\left|E[F_0](0,f^{(2)}) \right|_{C^{2,\alpha}}\right)\cdot|v^{(1)}_{k-1}-v^{(2)}_{k-1}|_{C^{2,\alpha}}\\
&\hspace{3cm}+\frac{1}{2}  \left|E[F_0](0,f^{(1)}-f^{(2)})\right|_{C^{2,\alpha}}\\
\stackrel{\eqref{eq:7.3}}{\leq}& \frac{1}{2}\left(|v^{(1)}_{k-1}-v^{(2)}_{k-1} |_{C^{2,\alpha}}+ \left|E[F_0](0,f^{(1)}-f^{(2)})\right|_{C^{2,\alpha}} \right)
\end{align*}
Since $v^{(1)}_0=0=v^{(2)}_0$ we obtain the desired estimate.
\end{proof}

The following result shows the existence of a family of isometric embeddings that depends continuously on the time parameter, in the local context.


\begin{theorem}\label{satz:7.1}
Let $M$ be a smooth closed $n$-dimensional manifold and let $(g(\cdot,t))_{t\in[0,T]}$ be a smooth family of Riemannian metrics on $M$. Given a free isometric
embedding $F_0\in C^{\infty}(M,\mathbb{R}^q)$ of the Riemannian manifold $(M,g(\cdot,0))$, then for all $x\in M$ there exists an open neighborhood $U_x\subset M$ of $x$, a time $T_x\in (0,T]$ and a family of
free immersions $(F(\cdot,t))_{t\in [0,T_x]}\subset C^{\infty}(\overline{U}_x,\mathbb{R}^q)$ with $F\in C^{0}(\overline{U}_x \times[0,T_x],\mathbb{R}^q)$
and $F(\cdot,0)=F_0$ so that for all $t\in[0,T_x]$ the equality
\begin{equation*}
F(\cdot,t)^{\ast}(\delta)=g(\cdot,t)
\end{equation*}
holds on $U_x$.
\end{theorem}

\begin{proof}
Let $F_0\in C^{\infty}(M,\mathbb{R}^q)$ be a free isometric embedding of  $(M,g(\cdot,0))$.
Fix $x\in M$ and a smooth chart $\varphi: U\longrightarrow B_{1}(0)$ with $\varphi(x)=0$ as
well as $\psi\in C^{\infty}(B_1(0))$ so that $\left. \psi\right|_{\overline{B_{1/2}(0)}}= 1$
and $supp(\psi)\subset B_{3/4}(0)$. The mapping $\left. F_0\circ \varphi^{-1} \right|_{B_1(0)}\in C^{\infty}(\overline{B_1(0)},\mathbb{R}^q)$ will be also denoted by $F_0$. We define $\widehat{g}\in C^{\infty}(\overline{B_1(0)},\mathbb{R}^{\frac{n}{2}(n+1)})$ as
\begin{equation}\label{eq:7.17}
\widehat{g}(x,t):=(\widehat{g}_{ij}(x,t))_{1\leq i\leq j\leq n}=\psi(x)\cdot ( \,^{\varphi}g_{ij}(x,t)-\,^{\varphi}g_{ij}(x,0))_{1\leq i\leq j\leq n}
\end{equation}
Choose $a\in C^{\infty}_0(B_1(0))$ with $\left. a\right|_{B_{3/4}(0)}= 1$ so that $a^2\, \widehat{g}= \widehat{g}$ and let $T_x\in (0,T]$ so that for all $t\in [0,T_x]$
\begin{equation*}
\left\Vert E[F_0]\right\Vert_{2,\alpha}\left|E[F_0](0,\widehat{g}(\cdot,t)) \right|_{C^{2,\alpha}}\leq \vartheta(\alpha,a)
\end{equation*}
where $\vartheta>0$ is chosen according to Theorem \ref{satz:5.1} and Lemma \ref{lem:7.1}. Then, for all $t\in [0,T_x]$ there exists $v(\cdot,t)\in C^{\infty}(\overline{B_1(0)},\mathbb{R}^q)$ so that $u(\cdot,t):= a^2\, v(\cdot,t)\in C^{\infty}_0(B_1(0),\mathbb{R}^q) $ solves
\begin{align}\label{eq:7.16}
\begin{split}
&\partial_i (F_0(x)+u(x,t))\cdot \partial_j (F_0(x)+u(x,t))\stackrel{\hphantom{\eqref{eq:7.17}}}{=}\partial_i F_0(x)\cdot \partial_j F_0(x)+\widehat{g}_{ij}(x,t)\\
\stackrel{\eqref{eq:7.17}}{=}&\,^{\varphi}g_{ij}(x,0)+\psi(x)\cdot (\,^{\varphi}g_{ij}(x,t)-\,^{\varphi}g_{ij}(x,0))
\end{split}
\end{align}
for all $(x,t)\in B_1(0)\times [0,T_x]$ and $1\leq i\leq j\leq n$. 

For later analysis we mention that for each $t\in[0,T_x]$ the mapping $v(\cdot,t)\in C^{\infty}(\overline{B_1(0)},\mathbb{R}^q)$
is the $C^{2,\alpha}$-limit of the sequence $(v_k(\cdot,t))_{k\in\mathbb{N}}\subset C^{\infty}(\overline{B_1(0)},\mathbb{R}^q)$ where
\begin{align}\label{eq:7.15}
v_k(\cdot,t)=
\begin{cases}
0 & \text{if }  k=0\\
\Phi[F_0,a,\widehat{g}(\cdot,t)](v_{k-1}(\cdot,t))  &\text{if } k\geq 1
\end{cases}
\end{align}
which is stated in \eqref{eq:5.47}. Using \eqref{eq:5.48} we have the estimate:
\begin{equation}\label{eq:7.18}
\left|v_k(\cdot,t) \right|_{C^{2,\alpha}}\leq\left|E[F_0]\left(0,\widehat{g}(\cdot,t)\right)\right|_{C^{2,\alpha}}
\end{equation}
for all $t\in [0,T_x]$. Since $\left. \psi\right|_{\overline{B_{1/2}(0)}}= 1$ we may define $U_x:= \varphi^{-1}(B_{1/2}(0))$ and $(F(\cdot,t))_{t\in[0,T]}\subset C^{\infty}(U_x,\mathbb{R}^q)$ as follows
\begin{equation*}
F(y,t)=F_0(y)+u(\varphi(y),t) 
\end{equation*}
for all $y\in U_x$. The claim follows from \eqref{eq:7.16} and \eqref{eq:7.5}.
\end{proof}


\subsection{Regularity of the local solution}
It is our aim to prove \textit{maximal regularity} with respect to the time parameter. In order to examine the regularity of the local solution in Section \ref{sec:7.1} we analyze the approximations in \eqref{eq:7.15}.
In this section we adapt the estimates in Section \ref{sec:5.4} to the time-dependent situation.
The following Lemma adjusts Lemma \ref{eq:5.31} to the time-dependent situation.
As it turns out, the fact that $N_i[a]$ is ``quadratic'' with respect to $v$ (cf. \eqref{eq:5.2}), plays an important role, here. 

\begin{lem}\label{lem:7.2}
Let $a\in C_0^{\infty}(B_1(0))$ be fixed, then for all $v_1,v_2\in C^{\infty}(\overline{{B_1(0)}}\times [0,T],\mathbb{R}^q)$ we have the estimate
\begin{align}\label{eq:7.6}
\begin{split}
&| N_i[a](v_1(\cdot,t) )-N_i[a](v_2(\cdot,t)) |_{C^{0,\alpha}} \\
\leq &K(\alpha,a)\cdot \left(|v_1(\cdot,t) |_{C^{2,\alpha}}+|v_2(\cdot,t) |_{C^{2,\alpha}} \right)\cdot|v_1(\cdot,t)-v_2(\cdot,t) |_{C^{2,\alpha}}
\end{split}
\end{align}
for all $t\in[0,T]$. If $v\in C^{\infty}(\overline{{B_1(0)}}\times [0,T],\mathbb{R}^q)$ and $m,r\in\mathbb{N}$ then
\begin{align}\label{eq:7.7}
\begin{split}
|\partial_t^r N_i[a](v(\cdot,t)) |_{C^{m,\alpha}}\leq & K(\alpha,a)\cdot\sum_{s=0}^r{\binom{r}{s}|\partial_t^s  v(\cdot,t) |_{C^{m+2,\alpha}}|\partial_t^{r-s} v(\cdot,t) |_{C^{2,\alpha}}}\\
&+C(m,\alpha,a)\cdot\sum_{s=0}^r{\binom{r}{s}|\partial_t^s v(\cdot,t)|_{C^{m+1,\alpha}}}|\partial_t^{r-s} v(\cdot,t) |_{C^{m+1,\alpha}}
\end{split}
\end{align}
for all $t\in[0,T]$, $1\leq i\leq n$.
\end{lem}

For the sake of readability we sometimes suppress the time-dependency in our notation. 

\begin{proof}
\begin{align*}
&|N_i[a](v_1) -N_i[a](v_2) |_{C^{0,\alpha}}\\
\stackrel{\eqref{eq:5.2}}{\leq}&|2\partial_i a\, (\Delta v_1\cdot v_1-\Delta v_2\cdot v_2)|_{C^{0,\alpha}}+|a\,(\Delta v_1\cdot\partial_i v_1 -\Delta v_2\cdot\partial_i v_2)|_{C^{0,\alpha}}\\
\stackrel{\eqref{eq:A.2}}{\leq} &K_1(\alpha,a)\cdot| \Delta (v_1-v_2) \cdot v_1|_{C^{0,\alpha}}+K_1(\alpha,a)\cdot|\Delta v_2\cdot (v_1-v_2)|_{C^{0,\alpha}}\\
&+K_1(\alpha,a)\cdot|\Delta (v_1-v_2)\cdot\partial_i v_1|_{C^{0,\alpha}} +K_1(\alpha,a)\cdot|\Delta v_2\cdot\partial_i(v_1- v_2)|_{C^{0,\alpha}}\\
\stackrel{\eqref{eq:A.16}}{\leq}&K_1(\alpha,a)\cdot| \Delta (v_1-v_2)|_{C^{0,\alpha}} \cdot |v_1|_{C^{0,\alpha}}+K_1(\alpha,a)\cdot|\Delta v_2|_{C^{0,\alpha}}\cdot |v_1-v_2|_{C^{0,\alpha}}\\
&+K_1(\alpha,a)\cdot|\Delta (v_1-v_2)|_{C^{0,\alpha}}\cdot|\partial_i v_1|_{C^{0,\alpha}}+K_1(\alpha,a)\cdot|\Delta v_2|_{C^{0,\alpha}}\cdot|\partial_i(v_1- v_2)|_{C^{0,\alpha}}\\
\stackrel{\eqref{eq:A.12}}{\leq}& K(\alpha,a)\cdot \left(|v_1 |_{C^{2,\alpha}}+|v_2 |_{C^{2,\alpha}} \right)|v_1-v_2 |_{C^{2,\alpha}}
\end{align*}

We prove \eqref{eq:7.7} for all $m\geq 1$ and $r\in\mathbb{N}$:
\begin{align*}
&|\partial_t^r N_i[a](v)|_{C^{m,\alpha}}\\
\stackrel{\eqref{eq:5.2}}{=}&|2\partial_i a\, \partial_t^r(\Delta v\cdot  v)+a \partial_t^r(\Delta v\cdot \partial_i v)|_{C^{m,\alpha}}\\
\stackrel{\eqref{eq:A.7}}{\leq}  &\sum_{s=0}^r \binom{r}{s} |2\partial_i a\, \Delta \partial^{s}_t v \cdot \partial^{r-s}_t v|_{C^{m,\alpha}}+\sum_{s=0}^r \binom{r}{s}|a\,\Delta \partial^{s}_t v \cdot \partial_i \partial^{r-s}_t v|_{C^{m,\alpha}}\\
\stackrel{\eqref{eq:A.15}}{\leq}&\sum_{s=0}^r \binom{r}{s}|2\partial_i a\, \Delta \partial^{s}_t v|_{C^{0,\alpha}} |\partial^{r-s}_t v|_{C^{m,\alpha}}+\sum_{s=0}^r \binom{r}{s}|2\partial_i a\, \Delta \partial^{s}_t v|_{C^{m,\alpha}}|\partial^{r-s}_t v|_{C^{0,\alpha}}\\
&+C_1(m,\alpha)\sum_{s=0}^r \binom{r}{s}|2\partial_i a\, \Delta \partial^{s}_t v|_{C^{m-1,\alpha}} |\partial^{r-s}_t v|_{C^{m-1,\alpha}}\\
+&\sum_{s=0}^r \binom{r}{s} |a\,\Delta \partial^{s}_t v|_{C^{0,\alpha}}|\partial_i \partial^{r-s}_t  v|_{C^{m,\alpha}}+\sum_{s=0}^r \binom{r}{s}|a\,\Delta \partial^{s}_t v|_{C^{m,\alpha}}|\partial_i \partial^{r-s}_t v|_{C^{0,\alpha}}\\
&+C_1(m,\alpha) \sum_{s=0}^r \binom{r}{s} |a\,\Delta \partial^{s}_t v|_{C^{m-1,\alpha}} |\partial_i \partial^{r-s}_t v|_{C^{m-1,\alpha}}\\
\stackrel{\eqref{eq:A.13}}{\leq} &K(\alpha,a)\cdot\sum_{s=0}^r{\binom{r}{s}|\partial_t^s  v |_{C^{m+2,\alpha}}|\partial_t^{r-s} v |_{C^{2,\alpha}}}+C(m,\alpha,a)\cdot\sum_{s=0}^r{\binom{r}{s}|\partial_t^s v|_{C^{m+1,\alpha}}|\partial_t^{r-s} v |_{C^{m+1,\alpha}}}
\end{align*}
\end{proof}

If $v\in C^{\infty}(\overline{{B_1(0)}}\times [0,T],\mathbb{R}^q)$ then  $N_i[a](v)\in C^{\infty}(\overline{{B_1(0)}}\times [0,T])$ which is a consequence
of \eqref{eq:5.2}. The same statement is true if we replace $N_i[a](v)$ by the $R_{ij}[a]$-part of $M_{ij}[a]$ (cf. Lemma \ref{lem:5.2}).
To derive an analogous  statement for the whole operator $M_{ij}[a]$ we need to consider the $L_{ij}[a]$-part (cf. Lemma \ref{lem:5.1}).
This operator has the feature that the inverse of the Laplace operator occurs. 
Using \textit{difference quotients} (cf. \cite[p. 168]{gilbarg2001elliptic}) and the Sobolev embedding theorem (cf. \cite[Theorem 7.10]{gilbarg2001elliptic}) we conclude the following result which implies that  $M_{ij}[a](v)\in C^{\infty}(\overline{{B_1(0)}}\times [0,T])$.
\begin{lem}\label{lem:7.3}
Given $f\in C^{\infty}(\overline{{B_1(0)}}\times [0,T])$. Let 
u: $B_1(0)\times [0,T]\longrightarrow \mathbb{R}$ such that for all $t\in[0,T]$ the function $u(\cdot,t)\in C^{\infty}(\overline{B_1(0)})$ is the unique solution to:
\begin{align*}
\begin{cases}
\Delta u(\cdot,t)=f(\cdot,t) &\text{on }B_1(0)\\
\hphantom{\Delta} u(\cdot,t)=0 &\text{on }\partial B_1(0)
\end{cases}
\end{align*}
then  $u\in C^{\infty}(\overline{{B_1(0)}}\times [0,T])$.
\end{lem}

The following Lemma shows that the operator $M_{ij}[a]$ has the same properties as the operator $N_{i}[a]$ in Lemma \ref{lem:7.2}. Since, the $R_{ij}[a]$-part in 
$M_{ij}[a]$ is ``quadratic'' with respect to $v$, this statement seems to be evident, if we just consider this part of the operator.
The $L_{ij}[a]$-part needs to be considered differently.

\begin{lem}\label{lem:7.4}
Let $a\in C_0^{\infty}(B_1(0))$ be fixed, then for all $v_1,v_2\in C^{\infty}(\overline{{B_1(0)}}\times [0,T],\mathbb{R}^q)$ we have the estimate
\begin{align}\label{eq:7.9}
\begin{split}
&| M_{ij}[a](v_1(\cdot,t) )-M_{ij}[a](v_2(\cdot,t)) |_{C^{0,\alpha}} \\
\leq &K(\alpha,a)\cdot \left(|v_1(\cdot,t) |_{C^{2,\alpha}}+|v_2(\cdot,t) |_{C^{2,\alpha}} \right)\cdot|v_1(\cdot,t)-v_2(\cdot,t) |_{C^{2,\alpha}}
\end{split}
\end{align}
for all $t\in[0,T]$. If $v\in C^{\infty}(\overline{{B_1(0)}}\times [0,T],\mathbb{R}^q)$ and $m,r\in\mathbb{N}$ then
\begin{align}\label{eq:7.10}
\begin{split}
|\partial_t^r M_{ij}[a](v(\cdot,t)) |_{C^{m,\alpha}}\leq &K(\alpha,a)\cdot\sum_{s=0}^r{\binom{r}{s}|\partial_t^s  v(\cdot,t) |_{C^{m+2,\alpha}}|\partial_t^{r-s} v(\cdot,t) |_{C^{2,\alpha}}}\\
&+C(m,\alpha,a)\cdot\sum_{s=0}^r{\binom{r}{s}|\partial_t^s v(\cdot,t)|_{C^{m+1,\alpha}}|\partial_t^{r-s} v(\cdot,t) |_{C^{m+1,\alpha}}}
\end{split}
\end{align}
for all $t\in[0,T]$, $1\leq i\leq j\leq n$.
\end{lem}

\begin{proof}
The operator $M_{ij}[a]$ is the sum of the operator $L_{ij}[a]$ and the operator $R_{ij}[a]$ (cf. \eqref{eq:5.11}). $L_{ij}[a]$ is the sum of terms that have the shape $\partial^{s_1} a \ \partial^{s_2} (\Delta^{-1}N_l[a])$ where $|s_1|+|s_2|=3$ and $|s_2|\leq 2$. Hence
\begin{align*}
&|\partial^{s_1} a \ \partial^{s_2} (\Delta^{-1}N_l[a](v_1))-\partial^{s_1} a \ \partial^{s_2} (\Delta^{-1}N_l[a](v_2)) |_{C^{0,\alpha}}\\
\stackrel{\eqref{eq:A.2}}{\leq}&K_1(\alpha,a)\cdot |  \Delta^{-1}[N_l[a](v_1)-N_l[a](v_2)]  |_{C^{2,\alpha}}\\
\stackrel{\eqref{eq:5.20}}{\leq}&K_2(\alpha,a)\cdot |  N_l[a](v_1)-N_l[a](v_2)  |_{C^{0,\alpha}}\\
\stackrel{\eqref{eq:5.31}}{\leq}& K_3(\alpha,a)\cdot \left(|v_1 |_{C^{2,\alpha}}+|v_2 |_{C^{2,\alpha}} \right)|v_1-v_2 |_{C^{2,\alpha}}
\end{align*}
Now, we consider the operator $R_{ij}[a](v)$. This operator is a sum of terms that have the shape
\begin{equation*}
\partial^{s_1} a  \partial^{s_2} a \, \partial^{s_3}v\cdot \partial^{s_4}v
\end{equation*}
where $s_1,s_2,s_3,s_4\in\mathbb{N}^n$ with $\sum_{l=1}^4{|s_l|}=4$ and $|s_3|,|s_4|\leq 2$. In this situation we have
\begin{align*}
&|\partial^{s_1} a  \partial^{s_2} a \, \partial^{s_3}v_1\cdot \partial^{s_4}v_1-\partial^{s_1} a  \partial^{s_2} a \, \partial^{s_3}v_2\cdot \partial^{s_4}v_2 |_{C^{0,\alpha}}\\
\stackrel{\eqref{eq:A.9}}{\leq}& K_1(\alpha,a) \cdot \left(|\partial^{s_3}(v_1-v_2)\cdot \partial^{s_4}v_1|_{C^{0,\alpha}}+|\partial^{s_3}v_2\cdot \partial^{s_4}(v_1-v_2) |_{C^{0,\alpha}}\right)\\
\stackrel{\eqref{eq:A.16}}{\leq}& K_2(\alpha,a) \cdot \left(|\partial^{s_3}(v_1-v_2)|_{C^{0,\alpha}}|\partial^{s_4}v_1|_{C^{0,\alpha}}+|\partial^{s_3}v_2|_{C^{0,\alpha}} |\partial^{s_4}(v_1-v_2) |_{C^{0,\alpha}}\right)\\
\stackrel{\eqref{eq:A.12}}{\leq}&K_3(\alpha,a)\cdot \left(|v_1|_{C^{2,\alpha}}+|v_2|_{C^{2,\alpha}}\right) \cdot |v_1-v_2|_{C^{2,\alpha}}
\end{align*}
which proves 
\eqref{eq:7.9}. To prove \eqref{eq:7.10}, we focus on the case  $m\geq 1$. In order to estimate the expression $|\partial_t^r L_{ij}[a](v) |_{C^{m,\alpha}}$ we consider
\begin{align*}
&|\partial_t^r \partial^{s_1} a \ \partial^{s_2} (\Delta^{-1}N_l[a](v)) |_{C^{m,\alpha}}=|\partial^{s_1} a \ \partial^{s_2} (\Delta^{-1}\partial_t^r N_l[a](v)) |_{C^{m,\alpha}}\\
\stackrel{\eqref{eq:A.13}}{\leq} &K_1(\alpha,a)\cdot |\partial^{s_2} (\Delta^{-1}\partial_t^r N_l[a](v)) |_{C^{m,\alpha}}+C_1(m,\alpha,a) \cdot |\Delta^{-1}\partial_t^r N_l[a](v) |_{C^{2,\alpha}}\\
&+C_1(m,\alpha,a)  \cdot |\Delta^{-1}\partial_t^r N_l[a](v) |_{C^{m+1,\alpha}}\\
\stackrel{\eqref{eq:5.22}}{\leq} &K_1(\alpha,a)\cdot |\partial^{s_2} (\Delta^{-1}\partial_t^r N_l[a](v)) |_{C^{m,\alpha}}+C_2(m,\alpha,a)  \cdot |\partial_t^r N_l[a](v) |_{C^{m-1,\alpha}}\\
\stackrel{\eqref{eq:A.6}}{\leq} &K_2(\alpha,a)\cdot |\Delta^{-1}\partial_t^r N_l[a](v)|_{C^{m+|s_2|,\alpha}}+K_2(\alpha,a)\cdot |\Delta^{-1}\partial_t^r N_l[a](v)|_{C^{2,\alpha}}\\
&+C_2(m,\alpha,a)  \cdot |\partial_t^r N_l[a](v) |_{C^{m-1,\alpha}}\\
\stackrel{\eqref{eq:5.22}}{\leq} &K_2(\alpha,a)\cdot |\Delta^{-1}\partial_t^r N_l[a](v)|_{C^{m+|s_2|,\alpha}}+C_3(m,\alpha,a)  \cdot |\partial_t^r N_l[a](v) |_{C^{m-1,\alpha}}\\
\stackrel{\eqref{eq:7.7}}{\leq} &K_2(\alpha,a)\cdot |\Delta^{-1}\partial_t^r N_l[a](v)|_{C^{m+|s_2|,\alpha}}\\
&+C_4(m,\alpha,a) \cdot \sum_{s=0}^r{\binom{r}{s}|\partial_t^s v(\cdot,t)|_{C^{m+1,\alpha}}}|\partial_t^{r-s} v(\cdot,t) |_{C^{m+1,\alpha}}
\end{align*}
If $|s_1|=1$ then \eqref{eq:5.22} and \eqref{eq:7.7} imply
\begin{align*}
&|\partial_t^r \partial^{s_1} a \ \partial^{s_2} (\Delta^{-1}N_l[a](v)) |_{C^{m,\alpha}}\leq C_5(m,\alpha,a) \cdot \sum_{s=0}^r{\binom{r}{s}|\partial_t^s v(\cdot,t)|_{C^{m+1,\alpha}}}|\partial_t^{r-s} v(\cdot,t) |_{C^{m+1,\alpha}}
\end{align*}
and if $|s_2|=2$ then
\begin{align*}
&|\Delta^{-1}\partial_t^r N_l[a](v)|_{C^{m+|s_2|,\alpha}}=|\Delta^{-1}\partial_t^r N_l[a](v)|_{C^{m+2,\alpha}}\\
\stackrel{\eqref{eq:5.30}}{\leq}&K_2(\alpha,a)\cdot|\partial_t^r N_l[a](v)|_{C^{m,\alpha}}+C_6(m,\alpha,a)\cdot|\partial_t^r N_l[a](v)|_{C^{m-1,\alpha}}
\end{align*}
Using \eqref{eq:7.7}, the desired estimate for the $L_{ij}$[a]-part follows.

It remains to prove the estimate for the summands of $R_{ij}[a](v)$. Let $s_1,s_2,s_3,s_4\in\mathbb{N}^n$ with $\sum_{l=1}^4{|s_l|}=4$ and $|s_3|,|s_4|\leq 2$, then
\begin{align*}
&|\partial_t^r(\partial^{s_1} a \partial^{s_2} a \, \partial^{s_3} v\cdot \partial^{s_4} v )|_{C^{m,\alpha}}\\
\stackrel{\eqref{eq:A.13}}{\leq} & K_1(\alpha,a)\cdot |\partial_t^r(\partial^{s_3} v\cdot \partial^{s_4} v) |_{C^{m,\alpha}}+C_1(m,\alpha,a)\cdot|\partial_t^r(\partial^{s_3} v\cdot \partial^{s_4} v) |_{C^{0,\alpha}}\\
&+C_1(m,\alpha)\cdot|\partial_t^r(\partial^{s_3} v\cdot \partial^{s_4} v)|_{C^{m-1,\alpha}}\\
\stackrel{\eqref{eq:A.6}}{\leq} & K_1(\alpha,a)\cdot |\partial_t^r(\partial^{s_3} v\cdot \partial^{s_4} v) |_{C^{m,\alpha}}+C_2(m,\alpha,a)\cdot|\partial_t^r(\partial^{s_3} v\cdot \partial^{s_4} v) |_{C^{m-1,\alpha}}\\
\stackrel{\eqref{eq:A.7}}{\leq} & K_1(\alpha,a)\cdot \sum_{s=0}^r{\binom{r}{s}| \partial^{s_3} \partial_t^s v\cdot \partial^{s_4} \partial_t^{r-s} v |_{C^{m,\alpha}}}\\
&+C_2(m,\alpha,a)\cdot\sum_{s=0}^r{\binom{r}{s}|\partial^{s_3} \partial_t^s v\cdot \partial^{s_4}\partial_t^{r-s} v |_{C^{m-1,\alpha}}}\\
\stackrel{\eqref{eq:A.16}}{\leq} & K_1(\alpha,a)\cdot \sum_{s=0}^r{\binom{r}{s}\left(| \partial^{s_3} \partial_t^s v|_{C^{0,\alpha}}\cdot |\partial^{s_4} \partial_t^{r-s} v |_{C^{m,\alpha}}+|\partial^{s_3} \partial_t^s v|_{C^{m,\alpha}}\cdot |\partial^{s_4} \partial_t^{r-s} v |_{C^{0,\alpha}}\right)}\\
&+C_3(m,\alpha,a)\cdot\sum_{s=0}^r{\binom{r}{s}|\partial^{s_3} \partial_t^{s} v|_{C^{m-1,\alpha}} |\partial^{s_4} \partial_t^{r-s}v |_{C^{m-1,\alpha}}}\\
\stackrel{\eqref{eq:A.12}}{\leq} & K_1(\alpha,a)\cdot \sum_{s=0}^r{\binom{r}{s}\left(| \partial^{s_3} \partial_t^s v|_{C^{0,\alpha}}\cdot |\partial^{s_4} \partial_t^{r-s} v |_{C^{m,\alpha}}+|\partial^{s_3} \partial_t^s v|_{C^{m,\alpha}}\cdot |\partial^{s_4} \partial_t^{r-s} v |_{C^{0,\alpha}}\right)}\\
&+C_3(m,\alpha,a)\cdot\sum_{s=0}^r{\binom{r}{s}|\partial_t^{s} v|_{C^{m+1,\alpha}} |\partial_t^{r-s}v |_{C^{m+1,\alpha}}}
\end{align*}
which implies the desired estimate for the $R_{ij}[a]$-part and consequently for the whole operator $M_{ij}[a]$.
\end{proof}

As it turns out, Lemma \ref{lem:7.2} and Lemma \ref{lem:7.3} imply a time-dependent version of Lemma \ref{lem:5.d}:
\begin{lem}\label{lem:7.5}
Let $a\in C_0^{\infty}(B_1(0))$ be fixed, then for all $v_1,v_2\in C^{\infty}(\overline{{B_1(0)}}\times [0,T],\mathbb{R}^q)$ we have:
\begin{align}\label{eq:7.11}
\begin{split}
&| P[a](v_1(\cdot,t) )-P[a](v_2(\cdot,t)) |_{C^{2,\alpha}}+| Q[a](v_1(\cdot,t) )-Q[a](v_2(\cdot,t)) |_{C^{2,\alpha}} \\
\leq &K(\alpha,a)\cdot \left(|v_1(\cdot,t) |_{C^{2,\alpha}}+|v_2(\cdot,t) |_{C^{2,\alpha}} \right)\cdot|v_1(\cdot,t)-v_2(\cdot,t) |_{C^{2,\alpha}}
\end{split}
\end{align}
for all $t\in[0,T]$. If $v\in C^{\infty}(\overline{{B_1(0)}}\times [0,T],\mathbb{R}^q)$, $m\geq 2$ and $r\in\mathbb{N}$ then
\begin{align}\label{eq:7.12}
\begin{split}
&|\partial_t^r P[a](v(\cdot,t)) |_{C^{m,\alpha}}+|\partial_t^r Q[a](v(\cdot,t)) |_{C^{m,\alpha}}\\
\leq &K(\alpha,a)\cdot\sum_{s=0}^r{\binom{r}{s}|\partial_t^s  v(\cdot,t) |_{C^{m,\alpha}}|\partial_t^{r-s} v(\cdot,t) |_{C^{2,\alpha}}}\\
&+C(m,\alpha,a)\cdot\sum_{s=0}^r{\binom{r}{s}|\partial_t^s v(\cdot,t)|_{C^{m-1,\alpha}}|\partial_t^{r-s} v(\cdot,t) |_{C^{m-1,\alpha}}}
\end{split}
\end{align}
for all $t\in[0,T]$.
\end{lem}

\begin{proof}
The estimates
\begin{align*}
&|P[a](v_1)-P[a](v_2) |_{C^{2,\alpha}}=\sum_{l=1}^n{|P_i[a](v_1)-P_i[a](v_2) |_{C^{2,\alpha}}}\\
\stackrel{\eqref{eq:5.14}}{=}&\sum_{i=1}^n{|a\Delta^{-1} \left[ N_i[a](v_1)- N_i[a](v_2)\right] |_{C^{2,\alpha}}}\\
\stackrel{\eqref{eq:A.9}}{\leq}&K_1(\alpha,a)\cdot\sum_{i=1}^n{|\Delta^{-1}\left[ N_i[a](v_1)- N_i[a](v_2)\right] |_{C^{2,\alpha}}}\\
\stackrel{\eqref{eq:5.20}}{\leq}&K_2(\alpha,a)\cdot| N_i[a](v_1)- N_i[a](v_2)|_{C^{0,\alpha}}\\
\stackrel{\eqref{eq:5.31}}{\leq}&K_3(\alpha,a)\cdot \left(|v_1 |_{C^{2,\alpha}} +|v_2 |_{C^{2,\alpha}}\right) |v_1-v_2 |_{C^{2,\alpha}}
\end{align*}
and
\begin{align*}
&|Q[a](v_1)-Q[a](v_2) |_{C^{2,\alpha}}=\sum_{1\leq i\leq j\leq n}{|Q_{ij}[a](v_1)-Q_{ij}[a](v_2) |_{C^{2,\alpha}}}\\
\stackrel{\eqref{eq:5.15}}{=}&\sum_{1\leq i\leq j\leq n}{|\Delta^{-1}\left[M_{ij}[a](v_1)-M_{ij}[a](v_2)\right] |_{C^{2,\alpha}}}\\
\stackrel{\eqref{eq:5.20}}{\leq}&\sum_{1\leq i\leq j\leq n}{K_4(\alpha,a)\cdot |M_{ij}[a](v_1)-M_{ij}[a](v_2)|_{C^{0,\alpha}}}\\
\stackrel{\eqref{eq:5.33}}{\leq}&K_5(\alpha,a)\cdot \left(|v_1 |_{C^{2,\alpha}} +|v_2 |_{C^{2,\alpha}}\right)|v_1-v_2 |_{C^{2,\alpha}}
\end{align*}
prove \eqref{eq:7.11}. We prove \eqref{eq:7.12} under the assumption that $m\geq 3$:
\begin{align*}
&|\partial_t^r P[a](v) |_{C^{m,\alpha}}=\sum_{i=1}^n{|\partial_t^r P_i[a](v)|_{C^{m,\alpha}}}\stackrel{\eqref{eq:5.14}}{=}\sum_{i=1}^n{|a\Delta^{-1} \partial_t^r N_i[a](v)|_{C^{m,\alpha}}}\\
\stackrel{\eqref{eq:A.13}}{\leq}&\sum_{i=1}^n{|a|_{C^{0,\alpha}}|\Delta^{-1} \partial_t^r N_i[a](v)|_{C^{m,\alpha}}}+\sum_{i=1}^n{|a|_{C^{m,\alpha}}|\Delta^{-1} \partial_t^r N_i[a](v)|_{C^{0,\alpha}}}\\
&+C_1(m,\alpha) \cdot\sum_{i=1}^n{|a|_{C^{m-1,\alpha}}|\Delta^{-1} \partial_t^r N_i[a](v)|_{C^{m-1,\alpha}}}\\
\stackrel{\hphantom{\eqref{eq:A.13}}}{\leq}&K_1(\alpha,a)\cdot\sum_{i=1}^n{|\Delta^{-1} \partial_t^r N_i[a](v)|_{C^{m,\alpha}}}+C_2(m,\alpha,a)\cdot\sum_{i=1}^n{|\Delta^{-1} \partial_t^r N_i[a](v)|_{C^{0,\alpha}}}\\
&+C_3(m,\alpha,a) \cdot\sum_{i=1}^n{|\Delta^{-1} \partial_t^r N_i[a](v)|_{C^{m-1,\alpha}}}\\
\stackrel{\eqref{eq:A.6}}{\leq}&K_1(\alpha,a)\cdot\sum_{i=1}^n{|\Delta^{-1} \partial_t^r N_i[a](v)|_{C^{m,\alpha}}}+C_4(m,\alpha,a) \cdot\sum_{i=1}^n{|\Delta^{-1} \partial_t^r N_i[a](v)|_{C^{m-1,\alpha}}}\\
\stackrel{ \eqref{eq:5.30}}{\leq} &K(\alpha,a)\cdot\sum_{i=1}^n{|\partial_t^r  N_i[a](v)|_{C^{m-2,\alpha}}}+C_5(m,\alpha,a)\cdot \sum_{i=1}^n{|\partial_t^r  N_i[a](v)|_{C^{m-3,\alpha}}}
\end{align*}
which implies the desired estimate. It remains to prove the desired estimate for the $\partial_t^r Q_{ij}[a](v)$-part.
\begin{align*}
&|\partial_t^r Q[a](v) |_{C^{m,\alpha}}=\sum_{1\leq i\leq j\leq n}{|\partial_t^r Q_{ij}[a](v)|_{C^{m,\alpha}}}=\sum_{1\leq i\leq j\leq n}{|\Delta^{-1}\partial_t^r M_{ij}[a](v)|_{C^{m,\alpha}}}\\
\leq & K(\alpha,a)\cdot\sum_{1\leq i\leq j\leq n}{|\partial_t^r M_{ij}[a](v)|_{C^{m-2,\alpha}}}+C(m,\alpha,a)\cdot\sum_{1\leq i\leq j\leq n}{|\partial_t^r M_{ij}[a](v)|_{C^{m-3,\alpha}}}
\end{align*}
The desired estimate follows from \eqref{eq:7.10}.
\end{proof}

We also need time-dependent adaptations of Lemma \ref{lem:5.a} and \ref{lem:5.e}:
\begin{lem}\label{lem:7.6}
Let $F_0\in C^{\infty}(\overline{B_1(0)},\mathbb{R}^q)$ be a free immersion and $r\in\mathbb{N}$. Then, for all
 $h\in C^{\infty}(\overline{B_1(0)}\times[0,T],\mathbb{R}^n)$ and $f\in C^{\infty}(\overline{B_1(0)}\times[0,T],\mathbb{R}^{\frac{n}{2}(n+1)})$ we have
\begin{align}\label{eq:7.13}
\left|\partial_t^r E[F_0](h(\cdot,t),f(\cdot,t)) \right|_{C^{m,\alpha}}\leq C(m,\alpha,F_0)\cdot\left(\left|\partial_t^r  h(\cdot,t) \right|_{C^{m,\alpha}}+\left|\partial_t^r  f(\cdot,t) \right|_{C^{m,\alpha}}\right)
\end{align}
for all $t\in [0,T]$.
\end{lem}

\begin{proof}
By definition, the $l$-th component of $E[F_0](h,f)$ (cf. \eqref{eq:5.19}) satisfies
\begin{equation}\label{eq:5.42}
E_l[F_0](h,f)=\sum_{i=1}^n{A_{l,i}h_i}+ \sum_{1\leq i\leq j\leq n}{B_{l,ij} f_{ij}}
\end{equation}
where $\{A_{l,i}\}_{1\leq i\leq n}\subset C^{\infty}(\overline{B_1(0)})$ and $\{B_{l,ij}\}_{1\leq i\leq j\leq n}\subset C^{\infty}(\overline{B_1(0)})$ depend on $F_0$.
Let $\beta\in\mathbb{N}^n$ be a multi index of order $k$ where $0\leq k\leq m$ then:
\begin{align*}
=&\left|\partial^{\beta} \partial_t^{r}(E_l[F_0](h,f))\right|_{C^{0,\alpha}}=\left|\partial^{\beta} (E_l[F_0](\partial_t^{r}h,\partial_t^{r}f))\right|_{C^{0,\alpha}}\\
=&\left|\sum_{i=1}^n{\partial^{\beta}(A_{l,i}\partial_t^{r} h_i)}+ \sum_{1\leq i\leq j\leq n}{\partial^{\beta}(B_{l,ij} \partial_t^{r} f_{ij})}\right|_{C^{0,\alpha}}\\
\stackrel{\eqref{eq:A.7}}{\leq}&\sum_{i=1}^n{\sum_{\gamma\leq \beta}{\binom{\beta}{\gamma}|\partial^{\gamma} A_{l,i}\partial^{\beta-\gamma}\partial_t^{r} h_i|_{C^{0,\alpha}}}}+ \sum_{1\leq i\leq j\leq n}{\sum_{\gamma\leq \beta}{\binom{\beta}{\gamma}|\partial^{\gamma} B_{l,ij} \partial^{\beta-\gamma}\partial_t^{r} f_{ij}|_{C^{0,\alpha}}}}\\
\stackrel{\eqref{eq:A.2}}{\leq}&\sum_{i=1}^n{\sum_{\gamma\leq \beta}{\binom{\beta}{\gamma}|A_{l,i}|_{C^{|\gamma|,\alpha}}|\partial_t^{r} h_i|_{C^{|\beta-\gamma|,\alpha}}}}+ \sum_{1\leq i\leq j\leq n}{\sum_{\gamma\leq \beta}{\binom{\beta}{\gamma}|B_{l,ij}|_{C^{|\gamma|,\alpha}} |\partial_t^{r} f_{ij}|_{C^{|\beta-\gamma|,\alpha}}}}\\
\stackrel{\eqref{eq:A.6}}{\leq}&C_1(k,\alpha,F_0)\cdot \left(|\partial_t^{r} h |_{C^{k,\alpha}}+|\partial_t^{r} f |_{C^{k,\alpha}}\right)
\end{align*}
which implies
\begin{align*}
&\left|\partial_t^{r} E_l[F_0](h,f)\right|_{C^{m,\alpha}}=\left| E_l[F_0](\partial_t^{r} h,\partial_t^{r} f)\right|_{C^{0,\alpha}}+\sum_{|\beta|=m}{\left|\partial^{\beta} (E_l[F_0](\partial_t^{r} h,\partial_t^{r} f)) \right|_{C^{0,\alpha}}}\\
\stackrel{\hphantom{\eqref{eq:A.12}}}{\leq}& C_1(\alpha,F_0)\cdot\left(|\partial_t^{r} h |_{C^{0,\alpha}}+|\partial_t^{r} f |_{C^{0,\alpha}}\right)+C_1(m,\alpha,F_0)\cdot\left(|\partial_t^{r} h |_{C^{m,\alpha}}+|\partial_t^{r} f |_{C^{m,\alpha}}\right)\\
\stackrel{\eqref{eq:A.12}}{\leq}&C_2(m,\alpha,F_0)\cdot\left(|\partial_t^{r} h |_{C^{m,\alpha}}+|\partial_t^{r} f |_{C^{m,\alpha}}\right)
\end{align*}
and finally
\begin{align*}
|\partial_t^{r} E[F_0](h,f) |_{C^{m,\alpha}}=&\sum_{l=1}^q{|\partial_t^{r} E_l[F_0](h,f)|_{C^{m,\alpha}}}\leq  C(m,\alpha,F_0)\cdot\left(|\partial_t^{r} h |_{C^{m,\alpha}}+|\partial_t^{r} f |_{C^{m,\alpha}}\right)
\end{align*}
\end{proof}
We also need an improved version of the previous Lemma:
\begin{lem}\label{lem:7.7}
Let $F_0\in C^{\infty}(\overline{B_1(0)},\mathbb{R}^q)$ be a free immersion and $r\in\mathbb{N}$. Let $m\geq 3$, then, for all $h\in C^{\infty}(\overline{B_1(0)}\times[0,T],\mathbb{R}^n)$ and $f\in C^{\infty}(\overline{B_1(0)}\times[0,T],\mathbb{R}^{\frac{n}{2}(n+1)})$ we have the estimate:
\begin{align}\label{eq:7.21}
\begin{split}
\left|\partial_t^r E[F_0](h(\cdot,t),f(\cdot,t)) \right|_{C^{m,\alpha}}\leq &\left\Vert E[F_0] \right\Vert_{2,\alpha}\cdot \left( \left|\partial_t^r h(\cdot,t)\right|_{C^{m,\alpha}}+\left| \partial_t^r f(\cdot,t) \right|_{C^{m,\alpha}}\right)\\
&+C(m,\alpha,F_0)\cdot\left(\left|\partial_t^r h(\cdot,t) \right|_{C^{m-1,\alpha}}+\left|\partial_t^r f(\cdot,t)\right|_{C^{m-1,\alpha}}\right)
\end{split}
\end{align}
for all $t\in [0,T]$.
\end{lem}

\begin{proof}
Let $l\in\{1,...,q\}$ and $\beta\in\mathbb{N}^n$ be a multi index of order $m-2$ then \eqref{eq:5.42} and
\eqref{eq:A.6} imply
\begin{align*}
&\left|\partial^{\beta} \partial_t^{r}E_l[F_0](h,f)-E_l[F_0](\partial^{\beta} \partial_t^{r}h,\partial^{\beta} \partial_t^{r}f) \right|_{C^{2,\alpha}}\\
\leq &C_1(\alpha)\cdot\sum_{0<\gamma\leq\beta}{\binom{\beta}{\gamma}\left(|\partial^{\gamma} A_{l,i}|_{C^{2,\alpha}}\cdot |\partial^{\beta-\gamma}\partial_t^{r}h_i|_{C^{2,\alpha}}+|\partial^{\gamma} B_{l,ij}|_{C^{2,\alpha}}\cdot|\partial^{\beta-\gamma} \partial_t^{r}f_{ij}|_{C^{2,\alpha}}\right)}\\
\leq &C_2(m,\alpha,F_0)\cdot \left(|\partial_t^{r}h |_{C^{m-1,\alpha}}+|\partial_t^{r}f |_{C^{m-1,\alpha}}\right)
\end{align*}
and hence
\begin{align}\label{eq:5.43}
\begin{split}
\left|\partial^{\beta} \partial_t^{r}E[F_0](h,f)-E[F_0](\partial^{\beta} \partial_t^{r}h,\partial^{\beta} \partial_t^{r}f) \right|_{C^{2,\alpha}}
\leq C_3(m,\alpha,F_0)\cdot \left(|\partial_t^{r}h |_{C^{m-1,\alpha}}+|\partial_t^{r}f |_{C^{m-1,\alpha}}\right)
\end{split}
\end{align}
Finally \eqref{eq:A.11} implies
\begin{align*}
&\left|E[F_0](\partial_t^{r}h,\partial_t^{r}f) \right|_{C^{m,\alpha}}\leq \left|E[F_0](\partial_t^{r}h,\partial_t^{r}f) \right|_{C^{0,\alpha}}+\sum_{|\beta|=m-2}{\left|\partial^{\beta} E[F_0](\partial_t^{r}h,\partial_t^{r}f) \right|_{C^{2,\alpha}
}}\\
\stackrel{\hphantom{\eqref{eq:5.43}}}{\leq} & \left\Vert E[F_0] \right\Vert_{0,\alpha}\cdot\left(\left|\partial_t^{r}h \right|_{C^{0,\alpha}}+\left|\partial_t^{r}f \right|_{C^{0,\alpha}}\right)\\
&+\sum_{|\beta|=m-2}{\left|\partial^{\beta} E[F_0](\partial_t^{r}h,\partial_t^{r}f)-E[F_0](\partial^{\beta} \partial_t^{r}h,\partial^{\beta} \partial_t^{r}f) \right|_{C^{2,\alpha}}}\\
&+\sum_{|\beta|=m-2}{\left|E[F_0](\partial^{\beta} \partial_t^{r}h,\partial^{\beta} \partial_t^{r}f) \right|_{C^{2,\alpha}}}\\
\stackrel{\eqref{eq:5.43}}{\leq} &\left\Vert E[F_0] \right\Vert_{0,\alpha}\cdot \left(|\partial_t^{r}h |_{C^{0,\alpha}}+|\partial_t^{r}f |_{C^{0,\alpha}}\right)+C(m,\alpha,F_0)\cdot\left(|\partial_t^{r}h |_{C^{m-1,\alpha}}+|\partial_t^{r}f |_{C^{m-1,\alpha}}\right)\\
&+\sum_{|\beta|=m-2}{\left\Vert E[F_0] \right\Vert_{2,\alpha}\cdot\left(|\partial^{\beta} \partial_t^{r}h |_{C^{2,\alpha}}+|\partial^{\beta} \partial_t^{r}f |_{C^{2,\alpha}}\right)}
\end{align*}
The desired estimate follows from \eqref{eq:A.12}.
\end{proof}

We use these considerations to iterate the regularity of the local solution in Theorem \ref{satz:7.1} on a smaller time interval, if necessary:

\begin{theorem}\label{satz:7.2}
Under the assumptions of Theorem \ref{satz:7.1} there exists $\widehat{T}_x\in (0,T_x]$ so that the solution $(F(\cdot,t))_{t\in [0,\widehat{T}_x]}\subset C^{\infty}(\overline{U}_x,\mathbb{R}^q)$ in Theorem \ref{satz:7.1} satisfies
\begin{equation*}
F\in C^{\infty}(\overline{U}_x\times[0,\widehat{T}_x],\mathbb{R}^q)
\end{equation*}
\end{theorem}

\begin{proof}
Lemma \ref{lem:7.3} implies that the approximation sequence $(v_k)_{k\in\mathbb{N}}$ in \eqref{eq:7.15} is smooth, i.e.: $(v_k)_{k\in\mathbb{N}}\subset C^{\infty}(\overline{B_1(0)}\times[0,T_x],\mathbb{R}^q)$ where we have used the smoothness of the metric $g$, the definition of the operator $\Phi[F_0,a,\widehat{g}]$ in \eqref{eq:5.18} and the definition of $P$ and $Q$ in \eqref{eq:5.14} and \eqref{eq:5.15}.
It is our aim to show that there exists $\widehat{T}_x\in (0,T_x]$ so that
\begin{align}\label{eq:7.26}
\left|\partial_t^r \partial^{\beta} v_k \right|_{C^0(\overline{B_1(0)}\times [0,\widehat{T}_x],\mathbb{R}^q)}\leq C(r,\beta)
\end{align}
for all $r,k\in\mathbb{N}$ and $\beta\in\mathbb{N}^n$ which implies $v\in C^{\infty}(\overline{B_1(0)}\times[0,\widehat{T}_x],\mathbb{R}^q)$.

Let $r\in\mathbb{N}$ then \eqref{eq:7.15}, \eqref{eq:5.18}, \eqref{eq:7.13}, \eqref{eq:7.12} and \eqref{eq:7.18} imply for all $k\geq 1$
\begin{align*}
 \left|\partial_t^r v_k(\cdot,t) \right|_{C^{2,\alpha}}\leq &2\cdot K_1(\alpha,a)\cdot   \left\Vert E[F_0] \right\Vert_{2,\alpha}\cdot\left|E[F_0]\left(0,\widehat{g}(\cdot,t)\right)\right|_{C^{2,\alpha}}|\partial_t^r v_{k-1}(\cdot,t) |_{C^{2,\alpha}}\\
&+K_1(\alpha,a)\cdot   \left\Vert E[F_0] \right\Vert_{2,\alpha}\cdot\sum_{s=1}^{r-1}{\binom{r}{s}|\partial_t^s  v_{k-1}(\cdot,t) |_{C^{2,\alpha}}|\partial_t^{r-s} v_{k-1}(\cdot,t) |_{C^{2,\alpha}}}\\
&\hspace{2cm}+\frac{1}{2}|E[F_0]\left(0,\partial_t^r \widehat{g}(\cdot,t) \right)|_{C^{2,\alpha}}
\end{align*}
for all $t\in [0,T_x]$. Assuming that $\widehat{T}_x\in(0,T_x]$ is sufficiently small so that
\begin{equation*}
\left\Vert E[F_0] \right\Vert_{2,\alpha}\cdot\left|E[F_0]\left(0,\widehat{g}(\cdot,t)\right)\right|_{C^{2,\alpha}}\leq \frac{1}{4K_1(\alpha,a)}
\end{equation*}
for all $t\in[0,\widehat{T}_x]$ then we obtain
\begin{align*}
|\partial_t^r v_k(\cdot,t) |_{C^{2,\alpha}}\leq &\frac{1}{2}\left(|\partial_t^r v_{k-1}(\cdot,t) |_{C^{2,\alpha}}+\left|E[F_0]\left(0,\partial_t^r \widehat{g}(\cdot,t) \right)\right|_{C^{2,\alpha}} \right)\\
&+K_1(\alpha,a)\cdot   \left\Vert E[F_0] \right\Vert_{2,\alpha}\cdot\sum_{s=1}^{r-1}{\binom{r}{s}|\partial_t^s  v_{k-1}(\cdot,t) |_{C^{2,\alpha}}|\partial_t^{r-s} v_{k-1}(\cdot,t) |_{C^{2,\alpha}}}
\end{align*}
for all $t\in [0,\widehat{T}_x]$. By induction over $r\in\mathbb{N}$, using Lemma \ref{lem:7.8}, we obtain \eqref{eq:7.26} for all $|\beta|\leq 2$ and $r\in\mathbb{N}$. Let $m\geq 3$, assuming that \eqref{eq:7.26} holds for all $|\beta|\leq m-1$ and $r\in\mathbb{N}$. Then \eqref{eq:7.15}, \eqref{eq:5.18}, \eqref{eq:7.21} and \eqref{eq:7.12} imply for all $k\geq 1$ and $r\in\mathbb{N}$
\begin{align}\label{eq:7.23}
\begin{split}
|\partial_t^r  v_k(\cdot,t) |_{C^{m,\alpha}}\leq &K_2(\alpha,a) \cdot\left\Vert E[F_0] \right\Vert_{2,\alpha}|\partial_t^r v_{k-1}(\cdot,t) |_{C^{m,\alpha}}| v_{k-1}(\cdot,t) |_{C^{2,\alpha}}\\
&+K_2(\alpha,a) \cdot\left\Vert E[F_0] \right\Vert_{2,\alpha}\cdot\sum_{s=0}^{r-1}{|\partial_t^s v_{k-1}(\cdot,t) |_{C^{m,\alpha}}|\partial_t^{r-s} v_{k-1}(\cdot,t)|_{C^{2,\alpha}}}\\
&+C_1(m,\alpha,F_0)\cdot \sum_{s=0}^r{|\partial^s v_{k-1}(\cdot,t) |_{C^{m-1,\alpha}}|\partial^{r-s} v_{k-1}(\cdot,t) |_{C^{m-1,\alpha}}}\\
&\hspace{2cm}+\frac{1}{2}\left|E[F_0]\left(0,\partial_t^r\widehat{g}(\cdot,t) \right)\right|_{C^{m,\alpha}}
\end{split}
\end{align}
for all $t\in [0,\widehat{T}_x]$. Then \eqref{eq:7.18}, \eqref{eq:7.26}, \eqref{eq:7.23}
and the induction hypothesis imply
\begin{align}\label{eq:7.25}
\begin{split}
|\partial_t^r  v_k(\cdot,t) |_{C^{m,\alpha}}\leq &K_2(\alpha,a) \cdot\left\Vert E[F_0] \right\Vert_{2,\alpha}\left|E[F_0]\left(0,\widehat{g}(\cdot,t)\right)\right|_{C^{2,\alpha}}\left|\partial_t^r v_{k-1}(\cdot,t) \right|_{C^{m,\alpha}}\\
&+C_2\left(m,r,\alpha,a,F_0,\widehat{g}\right)\cdot\sum_{s=0}^{r-1}{|\partial_t^s v_{k-1}(\cdot,t)|_{C^{m,\alpha}}}\\
&+C_3\left(m,r,\alpha,a,F_0,\widehat{g}\right)+\frac{1}{2}\left|E[F_0]\left(0,\partial_t^r\widehat{g}(\cdot,t) \right)\right|_{C^{m,\alpha}}
\end{split}
\end{align}
for all $t\in [0,\widehat{T}_x]$. If $\widehat{T}_x\in (0,T_x]$ is small enough to that for all $t\in [0,\widehat{T}_x]$
\begin{equation*}
\left\Vert E[F_0] \right\Vert_{2,\alpha}\left|E[F_0]\left(0,\widehat{g}(\cdot,t)\right)\right|_{C^{2,\alpha}}\leq \frac{1}{2K_2(\alpha,a)}
\end{equation*}
Then, by induction over $r\in\mathbb{N}$, \eqref{eq:7.25} and Lemma \ref{lem:7.8} imply \eqref{eq:7.26} for all $|\beta|=m$ and $r\in\mathbb{N}$.

\end{proof}

\subsection{Construction of the global embedding}

\begin{proof}[Proof of Theorem \ref{globeinb}]
Let $F_0\in C^{\infty}(M,\mathbb{R}^q)$ be a free isometric embedding of the Riemannian manifold $(M,g(\cdot,0))$ and let $\left\{\varphi: U_i\longrightarrow B_{1}(0)\right\}_{1\leq i\leq m}$ be a family of smooth charts so that
$M=\bigcup_{i=1}^m{\varphi_i^{-1}(B_{1}(0))}$ and let $\left\{\psi_i\right\}_{1\leq i\leq m}\subset C^{\infty}(M)$ with $supp(\psi_i)\subset \varphi_i^{-1}(B_1(0))$ for all $1\leq i \leq m$ so that $\sum_{i=1}^m{\psi_i}= 1$. Then, we have for all $(x,t)\in M\times [0,T]$ the equation
\begin{equation}\label{eq:7.41}
g(x,t)=g(x,0)+\sum_{i=1}^m{\underbrace{\psi_i [g(x,t)-g(x,0)]}_{=: \widehat{g}^{(i)}(x,t)}}
\end{equation}

 Applying the arguments in
the proof of Theorem \ref{satz:7.1}/\ref{satz:7.2} we already know that there exists a $T_1\in (0,T]$ and a family of free embeddings $F_1\in C^{\infty}(M\times[0,T_1])$ so that for all $t\in [0,T_1]$ the equality
\begin{equation}\label{eq:7.27}
F_1(\cdot,t)^{\ast}(\delta)=g(\cdot,0)+\widehat{g}^{(1)}(\cdot,t)
\end{equation}
holds on $M$. As in \eqref{eq:7.26} the mapping is defined as
\begin{equation*}
F_1(x,t):=
\begin{cases}
F_0(x)+u^{(1)}(\varphi_1(x),t) &\text{if }x\in U_1 \\
F_0(x) &\text{else }
\end{cases} 
\end{equation*}
where $u^{(1)}\in C^{\infty}(\overline{B_1(0)}\times [0,T_1],\mathbb{R}^q)$ satisfies $supp(u^{(1)}(\cdot,t))\subset B_1(0)$ for all $t\in [0,T_1]$.

In the following let $\left. F_1(\cdot,t)\circ \varphi_2^{-1} \right|_{B_1(0)}\in C^{\infty}(\overline{B_1(0)},\mathbb{R}^q)$
be also denoted by  $F_1(\cdot,t)$ for all $t\in [0,T_1]$. Let $a_2\in C^{\infty}_0(B_1(0))$ so that $\left. a_2 \right|_{supp(\psi_2\circ \varphi_2^{-1})}= 1$ and let $T_2\in (0,T_1]$ so that
\begin{equation*}
\left\Vert E[F_1(\cdot,t)]\right\Vert_{2,\alpha}\left|E[F_1(\cdot,t)](0,\,^{\varphi_2}\widehat{g}^{(2)}(\cdot,t)) \right|_{C^{2,\alpha}}\leq \vartheta(\alpha,a_2)
\end{equation*}
for all $t\in[0,T_2]$ where $\vartheta>0$ is chosen according to Theorem \ref{satz:5.1}. Then for all $t\in [0,T_2]$ there exists $v^{(2)}(\cdot,t)\in C^{\infty}(\overline{B_1(0)},\mathbb{R}^q)$ so that $u^{(2)}(\cdot,t):= a_2^2\, v^{(2)}(\cdot,t)\in C^{\infty}_0(B_1(0),\mathbb{R}^q)$ solves
\begin{align}\label{eq:7.29}
\begin{split}
&\partial_i (F_1(x,t)+u^{(2)}(x,t))\cdot \partial_j (F_1(x,t)+u^{(2)}(x,t))\\
\stackrel{\hphantom{\eqref{eq:7.17}}}{=}&\partial_i F_1(x,t)\cdot \partial_j F_1(x,t)+\,^{\varphi_2}\widehat{g}^{(2)}_{ij}(x,t)\stackrel{\eqref{eq:7.27}}{=}\, ^{\varphi_2}g_{ij}(x,0)+\,^{\varphi_2}\widehat{g}^{(1)}_{ij}(x,t)+\,^{\varphi_2}\widehat{g}^{(2)}_{ij}(x,t)
\end{split}
\end{align}
for all $x\in B_1(0)$ and $1\leq i\leq j\leq n$. Using \eqref{eq:5.47} for each $t\in [0,T_2]$, the mapping $v^{(2)}(\cdot,t)\in C^{\infty}(\overline{B_1(0)},\mathbb{R}^q)$
is the $C^{2,\alpha}$-limit of the sequence $(v^{(2)}_k(\cdot,t))_{k\in\mathbb{N}}\subset C^{\infty}(\overline{B_1(0)},\mathbb{R}^q)$ where
\begin{align}\label{eq:7.28}
v^{(2)}_k(\cdot,t)=
\begin{cases}
0 & \text{if }  k=0\\
\Phi[F_1(\cdot,t),a_2,\,^{\varphi_2}\widehat{g}^{(2)}(\cdot,t)](v^{(2)}_{k-1}(\cdot,t))  &\text{if } k\geq 1
\end{cases}
\end{align}
Lemma \ref{lem:7.3}, \eqref{eq:5.18} \eqref{eq:5.14} and \eqref{eq:5.15} imply that $(v^{(2)}_k)_{k\in\mathbb{N}}\subset C^{\infty}(\overline{B_1(0)}\times[0,T_2],\mathbb{R}^q)$ and \eqref{eq:5.48} implies
\begin{equation}\label{eq:7.30}
|v^{(2)}_k(\cdot,t) |_{C^{2,\alpha}}\leq \left|E[F_1(\cdot,t)]\left(0,\,^{\varphi_2}\widehat{g}^{(2)}(\cdot,t)\right)\right|_{C^{2,\alpha}}
\end{equation}
We define $(F_2(\cdot,t))_{t\in[0,T_2]}\subset C^{\infty}(M,\mathbb{R}^q)$ as follows
\begin{equation}\label{eq:7.40}
F_2(x,t):=
\begin{cases}
F_1(x,t)+u^{(2)}(\varphi_2(x),t) &\text{if }x\in U_2 \\
F_1(x,t) &\text{else}
\end{cases} 
\end{equation}
Then \eqref{eq:7.29} implies
\begin{equation*}
F_2(\cdot,t)^{\ast}(\delta)=g(\cdot,0)+\widehat{g}^{(1)}(\cdot,t)+\widehat{g}^{(2)}(\cdot,t)
\end{equation*}
for all $t\in [0,T_2]$. It is our aim to prove that $F_2\in C^{\infty}(M\times [0,T_2],\mathbb{R}^q)$ if $T_2\in (0,T_1]$
is small enough. In order to prove it, we show that for all $r\in\mathbb{N}$ and $\beta\in\mathbb{N}^n$ there exists a constant $C(r,\beta)>0$ so that for all $k\in\mathbb{N}$ the estimate
\begin{align}\label{eq:7.34}
|\partial_t^r \partial^{\beta} v^{(2)}_k |_{C^0(\overline{B_1(0)}\times [0,T_2],\mathbb{R}^q)}\leq C(r,\beta)
\end{align}
holds. In this step we need to pay regard to the \textit{time-dependency} of the operator $E[F_1(\cdot,t)]$ (cf. \eqref{eq:5.19}). Let
\begin{align*}
\partial_t^r E[F_1(\cdot,t)]: C^{\infty}(\overline{B_1(0)},\mathbb{R}^n) &\times C^{\infty}(\overline{B_1(0)},\mathbb{R}^{\frac{n}{2}(n+1)})\longrightarrow C^{\infty}(\overline{B_1(0)},\mathbb{R}^q)\\
\partial_t^r E[F_1(\cdot,t)](h,f)(x)&:= \partial_t^r\Theta[F_1(\cdot,t)](x)\cdot\begin{pmatrix} h(x)\\ f(x) \end{pmatrix}
\end{align*}
for all $t\in[0,T_2]$, where $\Theta[F_1(\cdot,t)]$ is defined in \eqref{eq:5.41}. The estimate
\begin{align}\label{eq:7.37}
\begin{split}
|\partial_t^r E[F_1(\cdot,t)](h,f)|_{C^{m,\alpha}}\leq  C(m,\alpha,r,F_1)\cdot\left(|h|_{C^{m,\alpha}}+|f|_{C^{m,\alpha}}\right)
\end{split}
\end{align}
is analogous to Lemma \ref{lem:7.6}. In order to prove the following estimate we point out that
\begin{equation}\label{eq:7.36}
\partial_t^r( E[F_1(\cdot,t)](h(\cdot,t),f(\cdot,t)))=\sum_{s=0}^r{\binom{r}{s}\partial_t^s E[F_1(\cdot,t)](\partial_t^{r-s} h(\cdot,t),\partial_t^{r-s} f(\cdot,t))}
\end{equation}
for all $t\in [0,T_2]$ where $h\in C^{\infty}\left(\overline{B_1(0)}\times [0,T_2],\mathbb{R}^n\right)$ and $f\in C^{\infty}\left(\overline{B_1(0)}\times [0,T_2],\mathbb{R}^{\frac{n}{2}(n+1)}\right)$.
Then \eqref{eq:7.28}, \eqref{eq:5.18}, \eqref{eq:7.36}, \eqref{eq:7.37}, \eqref{eq:7.12} and \eqref{eq:7.30} imply for all $r\in\mathbb{N}$ and $k\geq 1$
\begin{align*}
\begin{split}
& |\partial_t^r v^{(2)}_k(\cdot,t)|_{C^{2,\alpha}}\\
\leq&2 K_1(\alpha,a_2)\cdot\left\Vert  E[F_1(\cdot,t)]\right\Vert_{2,\alpha}\left|E[F_1(\cdot,t)]\left(0,\,^{\varphi_2}\widehat{g}^{(2)}(\cdot,t)\right)\right|_{C^{2,\alpha}}|\partial_t^{r} v^{(2)}_{k-1}(\cdot,t) |_{C^{2,\alpha}}\\
&+K_1(\alpha,a_2)\cdot\left\Vert  E[F_1(\cdot,t)] \right\Vert_{2,\alpha}\cdot\sum_{s=1}^{r-1}{\binom{r}{s}|\partial_t^{s}  v^{(2)}_{k-1}(\cdot,t) |_{C^{2,\alpha}}|\partial_t^{r-s} v^{(2)}_{k-1}(\cdot,t) |_{C^{2,\alpha}}}\\
&+K_1(\alpha,a_2)\cdot \sum_{s=1}^r   \left\Vert \partial_t^s E[F_1(\cdot,t)] \right\Vert_{2,\alpha}\cdot\sum_{\widehat{s}=0}^{r-s}{\binom{r-s}{\widehat{s}}|\partial_t^{\widehat{s}}  v^{(2)}_{k-1}(\cdot,t) |_{C^{2,\alpha}}|\partial_t^{r-s-\widehat{s}} v^{(2)}_{k-1}(\cdot,t)|_{C^{2,\alpha}}}\\
&+C\left(r,\alpha,a_2,F_1,\,^{\varphi_2}\widehat{g}^{(2)}\right)
\end{split}
\end{align*}
for all $t\in [0,T_2]$. We may assume that $T_2\in (0,T_1]$ is small enough, so that
\begin{equation*}
\left\Vert  E[F_1(\cdot,t)]\right\Vert_{2,\alpha}\left|E[F_1(\cdot,t)]\left(0,\,^{\varphi_2}\widehat{g}^{(2)}(\cdot,t)\right)\right|_{C^{2,\alpha}}\leq \frac{1}{4 K_1(\alpha,a_2)}
\end{equation*}
for all $t\in [0,T_2]$. Using Lemma \ref{lem:7.8} we infer estimate \eqref{eq:7.34} for all $|s|\leq 2$ and $r\in\mathbb{N}$
by induction over $r\in\mathbb{N}$.

Let $m\geq 3$, under the assumption that \eqref{eq:7.34} holds for all $|s|\leq m-1$ and $r\in\mathbb{N}$. Then for all $r\in\mathbb{N}$ and $k\geq 1$;  \eqref{eq:7.28}, \eqref{eq:5.18}, \eqref{eq:7.36}, \eqref{eq:7.13}, \eqref{eq:7.12} and \eqref{eq:7.37} imply
\begin{align*}
\begin{split}
& |\partial_t^r v^{(2)}_k(\cdot,t) |_{C^{m,\alpha}}\\
\leq&K_2(\alpha,a_2)\cdot\left\Vert  E[F_1(\cdot,t)] \right\Vert_{2,\alpha}\left|E[F_1(\cdot,t)]\left(0,\,^{\varphi_2}\widehat{g}^{(2)}(\cdot,t)\right)\right|_{C^{2,\alpha}}\cdot |\partial_t^r  v^{(2)}_{k-1}(\cdot,t) |_{C^{m,\alpha}}\\
&+C\left(m,r,\alpha,a_2,F_1,\,^{\varphi_2}\widehat{g}^{(2)}\right)\cdot\sum_{s=0}^{r-1}{|\partial_t^s  v^{(2)}_{k-1}(\cdot,t) |_{C^{m,\alpha}}|\partial_t^{r-s} v^{(2)}_{k-1}(\cdot,t) |_{C^{2,\alpha}}}\\
&+C\left(m,r,\alpha,a_2,F_1,\,^{\varphi_2}\widehat{g}^{(2)}\right)\cdot\sum_{s=0}^r{|\partial_t^s v^{(2)}_{k-1}(\cdot,t)|_{C^{m-1,\alpha}}|\partial_t^{r-s} v^{(2)}_{k-1}(\cdot,t) |_{C^{m-1,\alpha}}}\\
&+C\left(m,r,\alpha,a_2,F_1,\,^{\varphi_2}\widehat{g}^{(2)}\right)
\end{split}
\end{align*}
for all $t\in [0,T_2]$. If $T_2 \in (0,T_1]$ is chosen small enough so that
\begin{align*}
\left\Vert  E[F_1(\cdot,t)] \right\Vert_{2,\alpha} \left|E[F_1(\cdot,t)]\left(0,\,^{\varphi_2}\widehat{g}^{(2)}(\cdot,t)\right)\right|_{C^{2,\alpha}}\leq \frac{1}{2 K_2(\alpha,a_2)}
\end{align*}
then, by induction over $r\in \mathbb{N}$, Lemma \ref{lem:7.8}  implies \eqref{eq:7.34} for all $|s|=m$ and all $r\in \mathbb{N}$.

The construction \eqref{eq:7.40} may be applied iteratively to the charts $\left\{\varphi_i: U_i\longrightarrow B_{1}(0)\right\}_{3\leq i\leq m}$ 
which proves the claim, using decomposition \eqref{eq:7.41}.

\end{proof}

\appendix
\section{H\"older spaces}
Let $C^{m,\alpha}(B_1(0))$ be the H\"older-space where $m\in\mathbb{N}$ and $\alpha\in (0,1)$. We define
\begin{equation*}
 |u|_{C^{0,\alpha}}:=\sup_{x\in B_1(0)}|u(x)|+\sup_{x,y\in B_1(0), x\neq y}\left\{\frac{|u(x)-u(y)|}{|x-y|^{\alpha}} \right\}
\end{equation*}
and if $m\geq 1$
\begin{align*}
\left|u\right|_{C^{m,\alpha}}:=| u |_{C^{0,\alpha}}+\sum_{|s|=m}{|\partial^s u|_{C^{0,\alpha}}} 
\end{align*}
If $u,v\in C^{0,\alpha}(B_1(0))$ we have
\begin{equation}\label{eq:A.2}
|uv|_{C^{0,\alpha}}\leq |u|_{C^{0,\alpha}}|v|_{C^{0,\alpha}}
\end{equation}
If $k\in\mathbb{N}$ and $\beta\in (0,1)$ such that $k+\beta<m+\alpha$
then
\begin{equation}\label{eq:A.6}
\left| u \right|_{C^{k,\beta}}\leq C(m,k,\alpha,\beta)\cdot\left| u \right|_{C^{m,\alpha}}
\end{equation}
The next estimate follows from the formula
\begin{equation}\label{eq:A.7}
\partial^{\beta}(uv)=\sum_{\gamma\leq \beta}{\binom{\beta}{\gamma}}\partial^{\gamma}u\ \partial^{\beta-\gamma}v
\end{equation}
Let $u,v\in C^{m,\alpha}(B_1(0))$ then \eqref{eq:A.2} and \eqref{eq:A.6} imply
\begin{align}\label{eq:A.8}
\begin{split}
&\left|uv \right|_{C^{m,\alpha}}=\left|uv \right|_{C^{0,\alpha}}+\sum_{|s|=m}{\left|\partial^s(uv) \right|_{C^{0,\alpha}}} \\
\leq&\left|u\right|_{C^{0,\alpha}}\left|v\right|_{C^{m,\alpha}}+\left|u\right|_{C^{m,\alpha}}\left|v\right|_{C^{0,\alpha}}+C(m,\alpha)\cdot\left|u\right|_{C^{m-1,\alpha}}\left|v\right|_{C^{m-1,\alpha}}
\end{split}
\end{align}
and
\begin{equation}\label{eq:A.9}
\left|uv \right|_{C^{m,\alpha}}\leq C(m,\alpha)\cdot \left|u\right|_{C^{m,\alpha}}\left|v\right|_{C^{m,\alpha}}
\end{equation}
as well as 
\begin{align}\label{eq:A.10}
\begin{split}
\left|u \right|_{C^{m,\alpha}}\leq C(\alpha)\cdot\left| u \right|_{C^{2,\alpha}}+\sum_{|s|=k}{\left|\partial^s u \right|_{C^{m-k,\alpha}}} 
\end{split}
\end{align} 
for all $1\leq k\leq m$. We also introduce the following H\"older-spaces of vector valued functions
\begin{align*}
C^{m,\alpha}(B_1(0),\mathbb{R}^q):=\bigl\{(u_1,...,u_q)^{\top}: B_1(0)\longrightarrow \mathbb{R}^q\ |\ u_j\in C^{m,\alpha}(B_1(0))\ \forall 1\leq j\leq q \bigr\}
\end{align*}
with the norm
\begin{align*}
|u|_{C^{m,\alpha}}:= \sum _{j=1}^q{\left|u_j \right|_{C^{m,\alpha}}}
\end{align*}
For all $k\in\mathbb{N}$ and $\beta\in (0,1)$ with $k+\beta<m+\alpha$ we have
\begin{equation}\label{eq:A.12}
\left|u\right|_{C^{k,\beta}}\leq C(m,k,\alpha,\beta)\cdot\left|u\right|_{C^{m,\alpha}}
\end{equation}
If $a\in C^{m,\alpha}(B_1(0))$ and $u\in C^{m,\alpha}(B_1(0),\mathbb{R}^q)$ then \eqref{eq:A.8} implies
\begin{align}\label{eq:A.13}
\begin{split}
\left|au\right|_{C^{m,\alpha}}\leq \left| a \right|_{C^{0,\alpha}}\left| u \right|_{C^{m,\alpha}}+\left| a \right|_{C^{m,\alpha}}\left| u \right|_{C^{0,\alpha}}+C(m,\alpha)\cdot \left| a \right|_{C^{m-1,\alpha}}\left| u \right|_{C^{m-1,\alpha}}
\end{split}
\end{align}
\eqref{eq:A.6} and \eqref{eq:A.12} imply
\begin{equation*}
\left|au\right|_{C^{m,\alpha}}\leq C(m,\alpha)\cdot \left| a \right|_{C^{m,\alpha}}\left| u \right|_{C^{m,\alpha}}
\end{equation*}
If $v\in C^{m,\alpha}(B_1(0),\mathbb{R}^q)$ then \eqref{eq:A.8} implies
\begin{align}\label{eq:A.15}
\left|u\cdot v\right|_{C^{m,\alpha}}\leq\left|u \right|_{C^{0,\alpha}}\left|v \right|_{C^{m,\alpha}}+\left|u \right|_{C^{m,\alpha}}\left|v \right|_{C^{0,\alpha}}+C(m,\alpha)\cdot \left|u \right|_{C^{m-1,\alpha}}\left|v \right|_{C^{m-1,\alpha}}
\end{align}
and using \eqref{eq:A.12}
\begin{equation}\label{eq:A.16}
\left|u\cdot v\right|_{C^{m,\alpha}}\leq C(m,\alpha)\cdot \left|u \right|_{C^{m,\alpha}}\left|v \right|_{C^{m,\alpha}}
\end{equation}
Finally \eqref{eq:A.10} implies
\begin{align}\label{eq:A.11}
\begin{split}
\left|u \right|_{C^{m,\alpha}}\leq \left|u \right|_{C^{0,\alpha}}+\sum_{|s|=k}{\left|\partial^s u \right|_{C^{m-k,\alpha}}}
\end{split}
\end{align}
for all $1\leq k\leq m$.

\begin{lem}\label{lem:7.8}
Let $(a_k)_{k\in\mathbb{N}}\geq 0$ so that 
\begin{equation*}
a_{k+1}\leq \frac{a_k}{2}+C
\end{equation*}
for all $k\in\mathbb{N}$ where $C\geq 0$ does not depend on $k$, then the estimate
\begin{equation*}
a_{k}\leq a_0+2C
\end{equation*}
holds for all $k\in\mathbb{N}$.
\end{lem}

\begin{proof}
The estimate holds if $k=0$. Suppose the estimate is true for $k\in\mathbb{N}$ then
\begin{equation*}
a_{k+1}\leq \frac{a_k}{2}+C \leq \frac{a_0+2C}{2}+C\leq a_0+2C
\end{equation*}
\end{proof}

\bibliographystyle{alphaurl}
\bibliography{biblio}

\end{document}